\documentclass{siamart1116}
\usepackage{amsmath}
\usepackage{amsfonts}
\usepackage{amssymb}
\usepackage{algorithm}
\usepackage{algorithmic}
\usepackage{url}
\usepackage{subfigure}
\usepackage{titlesec}
\usepackage[titletoc]{appendix}
\usepackage{mathtools}

\def\noprint#1{}

\newtheorem{assumption}{Assumption}

\newcommand{\lambdamin}{\lambda_{\mbox{\rm\scriptsize{min}}}}
\newcommand{\lambdamax}{\lambda_{\mbox{\rm\scriptsize{max}}}}

\newcommand{\numin}{\nu_{\mbox{\rm\scriptsize{min}}}}

\newcommand{\R}{\mathbb{R}}

\newcommand{\cR}{\mathcal{R}}
\newcommand{\dist}{\mbox{\rm dist}}

\newcommand{\eps}{\epsilon}
\newcommand{\epsg}{\epsilon_g}
\newcommand{\epsH}{\epsilon_H}
\newcommand{\Glow}{G_{\mbox{\rm\scriptsize low}}}

\newcommand{\clocal}{c_{\text{\rm local}}}

\newcommand{\cgrad}{c_{\text{\rm grad}}}
\newcommand{\jgrad}{j_{\text{\rm grad}}}

\newcommand{\cnc}{c_{\text{\rm nc}}}
\newcommand{\jnc}{j_{\text{\rm nc}}}
\newcommand{\jbt}{\tilde{j}}
\newcommand{\cbt}{\tilde{c}}

\newcommand{\argmin}{\mbox{argmin}}

\newcommand{\cO}{{\cal O}}

\newcommand{\T}{\mathcal{T}}
\newcommand{\Kouter}{\mathcal{K}_{\mathrm{outer}}}
\newcommand{\Klarge}{\mathcal{K}_{\mathrm{large}}}
\newcommand{\Ktotal}{\mathcal{K}_{\mathrm{total}}}
\newcommand{\Klocal}{\mathcal{K}_{\mathrm{local}}}
\newcommand{\Cmeo}{\mathcal{C}_{\mathrm{meo}}}
\newcommand{\Nmeo}{N_{\mathrm{meo}}}
\newcommand{\rank}{\mathrm{rank}}
\newcommand{\muhalf}{\frac12}
\newcommand{\trace}{\mbox{trace}\,}
\newcommand{\calpha}{c_{\alpha}}
\newcommand{\cbeta}{c_{\beta}}
\newcommand{\cgamma}{c_{\gamma}}
\newcommand{\cepsilon}{c_{\epsilon}}
\newcommand{\mincacb}{\min\left\{1, \calpha^{3/2}\cbeta^{3/2}\right\}}

\newcommand{\cfac}{c_s}

\DeclarePairedDelimiter{\ceil}{\lceil}{\rceil}

\definecolor{darkgreen}{RGB}{34, 139, 34}

\title{A Line-Search Descent Algorithm for Strict Saddle Functions with Complexity Guarantees%
  \thanks{Version of \today. Research supported by NSF Awards 1628384, 1634597, and 1740707; Subcontract  8F-30039 from Argonne National Laboratory; and Award N660011824020 from the DARPA Lagrange Program.}}

\author{Michael O'Neill and Stephen J. Wright}

\begin{document}

\maketitle

\begin{abstract}
We describe a line-search algorithm which achieves the best-known worst-case complexity results for problems with a certain ``strict saddle'' property that has been observed 
to hold in low-rank matrix optimization problems. Our algorithm is adaptive, in the sense that it makes use of backtracking line searches and does not require prior knowledge of the parameters that define the strict saddle property.
\end{abstract}

\section{Introduction.}
\label{sec:intro}

Formulation of machine learning (ML) problems as nonconvex optimization problems has produced significant advances in several key areas.
While general nonconvex optimization is difficult, both in theory and in practice, the problems arising from ML applications often have structure that makes them  solvable by local descent methods. 
For example, for functions with the ``strict saddle'' property, nonconvex optimization methods can efficiently find local (and often global) minimizers \cite{JSun_QQu_JWright_2015}.

In this work, we design an optimization algorithm for a class of low-rank matrix problems that includes matrix completion, matrix sensing, and Poisson prinicipal component analysis.
Our method seeks a rank-$r$ minimizer of the function $f(X)$, where $f:\R^{n \times m} \to \R$.
The matrix $X$ is parametrized explicitly as the outer product of two matrices $U \in \R^{n \times r}$ and $V \in \R^{m \times r}$, where $r \le \min(m,n)$. We make use throughout of the notation
\begin{equation} \label{eq:W}
W = \left[ \begin{matrix} U \\ V \end{matrix} \right] \in \R^{(m+n)
  \times r}.
\end{equation}
The problem is reformulated in terms of $W$ and an objective function $F$ as follows:
\begin{equation} \label{eq:fdef}
  \min_{W} \, F(W) := f(UV^T), \quad \mbox{where $W$, $U$, $V$ are
    related as in \eqref{eq:W}.}
\end{equation}
Under suitable assumptions as well as the use of a specific regularizer, these problems obey the ``robust strict saddle property'' described by \cite{ZZhu_QLi_GTang_MBWakin_2017}.
This property divides the search space into three regions: one in which the gradient of $F$ is large, a second in which the Hessian of $F$ has a direction of large negative curvature, and a third that is a neighborhood of the solution set, inside which a local regularity condition holds.

In this paper, we describe and analyze an algorithm with a favorable worst-case complexity for this class of problems.
We characterize the maximum number of iterations as well as the maximum number of gradient evaluations required to find an approximate second order solution, showing that these quantities have at most a logarithmic dependence on the accuracy parameter for the solution.
This is a vast improvement over the worst-case complexity of methods developed for general smooth, nonconvex optimization problems, which have a polynomial dependence on the inverse of the accuracy parameter (see, for example, \cite{YCarmon_JCDuchi_OHinder_ASidford_2018,CWRoyer_MONeill_SJWright_2019}).

While other algorithms for optimizing robust strict saddle problems have previously been developed, knowledge of the strict saddle parameters is required to obtain a worst case complexity that depends at most logarithmically on the solution accuracy \cite{JSun_QQu_JWright_2015}.
For low-rank matrix problems, the strict saddle parameters depend on the singular values of the matrices of the optimal solution set \cite{ZZhu_QLi_GTang_MBWakin_2017}, and are thus unlikely to be known a-priori.
Therefore, an essential component of any implementable method for low-rank matrix problems is adaptivity to the optimization geometry, a property achieved by the method developed in this paper.
Our method maintains  an estimate of the key strict saddle parameter that is used to predict which of the three regions described above contains the current iterate. 
This prediction determines whether a negative gradient step or a negative curvature step is taken. 
When the method infers that the iterate is in a neighborhood of the solution set, a monitoring strategy is employed to detect fast convergence, or else flag that an incorrect prediction has been made. By reducing our parameter estimate after incorrect predictions, our method naturally adapts to the optimization landscape and achieves essentially the same behavior as an algorithm for which the critical parameter is known.

\paragraph{Notation and Background.}
We make use in several places of ``hat'' notation for matrices in the form of $W$ in \eqref{eq:W} in which the elements in the bottom half of the matrix are negated, that is
\begin{equation} \label{eq:hatW}
  \hat{W} := \left[ \begin{matrix} U \\ -V \end{matrix} \right].
\end{equation}

We use the notation $\langle A, B \rangle = \trace (A^\top B)$.

For a scalar function $h(Z)$ with a matrix variable $Z \in \R^{p \times q}$, the gradient $\nabla h(Z)$ is an $p \times q$ matrix whose $(i,j)$-th entry is $\frac{\partial h(Z)}{\partial Z_{i,j}}$ for all $i =1,2,\dotsc,p $ and $j=1,2,\dotsc,q$. 
In some places, we take the Hessian of $\nabla^2 h(Z)$ to be an $pq \times pq$ matrix whose $(i,j)$ element is $\frac{\partial^2 h(Z)}{\partial z_i  \partial z_j}$ where $z_i$ is the $i$-th coordinate of the vectorization of $Z$.  
In other places, the Hessian is represented as a bilinear form defined by $[\nabla^2 h(Z)] (A,B) = \sum_{i,j,k,l} \frac{\partial^2 h(Z)}{\partial Z_{i,j} Z_{k,l}} A_{i,j} B_{k,l}$ for any $A, B \in \R^{p\times q}$. 
We also write $\langle A, \nabla^2 h(Z) B \rangle$ for this same object.

Under these definitions, we can define the maximum and minimum eigenvalues of $\nabla^2 h(Z)$ as 
\begin{equation} \label{eq:lambdamaxmin}
\lambdamax(\nabla^2 h(Z)) = \max_D \frac{\langle D, \nabla^2 h(Z) D\rangle}{\|D\|_F^2}, \quad
\lambdamin(\nabla^2 h(Z)) =\\ \min_D \frac{\langle D, \nabla^2 h(Z) D\rangle}{\|D\|_F^2}.
\end{equation}

Occasionally we need to refer to the gradient and Hessian of the original function $f$, prior to reparameterization. We denote the gradient by $\nabla f(X)$ and the Hessian by $\nabla^2 f(X)$, where $X = U V^\top$. 

Our algorithm seeks a point that approximately satisfies second-order necessary optimality conditions for a regularized objective function $G:\R^{(m+n) \times r} \to \R$ to be defined later in (\ref{eq:gdef}), that is,
\begin{equation} \label{eq:optconds}
\|\nabla G(W)\|_F \leq \epsg, \quad \quad \lambdamin(\nabla^2 G(W)) \geq -\epsH,
\end{equation}
for small positive tolerances $\epsg$ and $\epsH$.

We assume that explicit storage and calculation of the Hessian $\nabla^2 G(W)$ is
undesirable, but that products of the form
$\nabla^2 G(W) Z$ can be computed efficiently for arbitrary matrices $Z \in \mathbb{R}^{(n+m)\times r}$. Computational differentiation techniques can be used to evaluate such products at a cost that is a small multiple of the cost of the gradient evaluation $\nabla G(W)$ \cite{AGriewank_AWalther_2008}.

\section{Related Work.} \label{sec:related}

One major class of algorithms that have been developed to solve low-rank matrix problems utilizes customized initialization procedures to find a starting point which lies in the basin of attraction of a global minimizer. 
A standard optimization procedure, such as gradient descent, initialized at this point, typically converges to the minimizer. 
Due to the local regularity of the function in a neighborhood around the solution set, many of these methods have a linear rate of convergence after the initialization step. 
However, the spectral methods used to find a suitable starting point are often computationally intensive and involve procedures such as reduced singular value decomposition and/or projection onto the set of rank $r$ matrices. 
These methods have been applied to a wide variety of problems including phase retrieval \cite{EJCandes_XLi_MSoltanolkotabi_2015}, blind-deconvolution \cite{XLi_SLing_TStrohmer_KWei_2019}, matrix completion \cite{RHKeshavan_AMontanari_SOh_2010, RSun_ZQLuo_2016}, and matrix sensing \cite{STu_RBoczar_MSimchowitz_MSoltanolkotabi_BRecht_2016}.

Another line of work focuses on characterizing the set of critical points of $f$.
Many low-rank recovery problems are shown to obey the strict saddle assumption, in which all saddle points of the function exhibit directions of negative curvature in the Hessian. 
Additionally, these problems often have the favorable property that all local minimizers are global minimizers.
Examples include dictionary learning \cite{JSun_QQu_JWright_2016}, phase retrieval \cite{JSun_QQu_JWright_2018}, tensor decomposition \cite{RGe_FHuang_CJin_YYuan_2015}, matrix completion \cite{RGe_CJin_YZheng_2017}, and matrix sensing \cite{SBhojanapalli_BNeyshabur_NSrebro_2016, RGe_CJin_YZheng_2017}.
When all these properties hold, gradient descent initialized at a random starting point converges in the limit, with probability 1, to a global minimizer \cite{JDLee_MSimchowitz_MIJordan_BRecht_2016}. 
Two recent works demonstrate that randomly initialized gradient descent has a global linear rate of convergence when applied to phase retrieval \cite{YChen_YChi_JFan_CMa_2019} and dictionary learning \cite{DGilboa_SBuchanan_JWright_2019}.

A number of works go a step further and characterize the global optimization geometry of these problems.
These problems satisfy the robust strict saddle property, in which the domain is partitioned such that any point is either in a neighborhood of a global solution, has a direction of sufficient negative curvature in the Hessian, or has a large gradient norm.
This property has been shown to hold for tensor decomposition \cite{RGe_FHuang_CJin_YYuan_2015}, phase retrieval \cite{JSun_QQu_JWright_2018}, dictionary learning \cite{JSun_QQing_JWright_2016}, and general low-rank matrix problems \cite{ZZhu_QLi_GTang_MBWakin_2017}. 
Due to the partitioning, methods developed for general non-convex problems with saddle point escaping mechanisms are of interest in thie context. Indeed, methods such as gradient descent with occasional perturbations to escape saddle points appear to converge at a global linear rate.
However, a close reading of these methods reveals that knowledge of the strict saddle parameters defining the separate regions of the domain is required to escape saddle points efficiently and obtain linear convergence rates. 
In particular, for gradient descent with perturbations, these parameters are used to decide when the perturbations should be applied. 
The same issue arises for methods developed specifically for strict saddle functions, such as the second-order trust region method of \cite{JSun_QQu_JWright_2015} and the Newton-based method of \cite{SPaternain_AMokhtari_ARibeiro_2019}, the latter requiring knowledge of a strict saddle parameter to flip the eigenvalues of the Hessian matrix at every iteration.
Unfortunately, for low-rank matrix problems of the form (\ref{eq:fdef}), these parameters correspond to the first  and $r$-th singular value of the optimal solution \cite{ZZhu_QLi_GTang_MBWakin_2017} --- information that is unlikely to be known a-priori.

In this work, we develop the first method for low-rank matrix problems whose worst-case complexity depends at most logarithmically on the solution accuracy, without relying on  an expensive initialization procedure or knowledge of the strict saddle parameters.
The method maintains its estimate of the crucial strict saddle parameter along with gradient and negative curvature information to infer which of the three regions in the partition mentioned above is occupied  by the current iterate.
By choosing appropriate steps based on this inference, the method converges to an approximate second-order stationary point from any starting point while dependence on the approximation tolerances in \eqref{eq:optconds} is only logarithmic.

\section{Robust Strict Saddle Property and Assumptions.} \label{sec:assumptions}

Here we provide the background and assumptions needed to describe the robust strict saddle property for low-rank matrix problems, as well as the additional assumptions required by our optimization algorithm.
Section~\ref{subsec:regularitycond} provides defitions for functions invariant under orthogonal transformations, our local regularity condition, and the robust strict saddle property. Section~\ref{subsec:regularization} discusses the regularization term that we add to $F(W)$ and provides definitions for the gradient and Hessian of the regularized function. 
Finally, we describe our assumptions and the strict saddle parameters in Section~\ref{subsec:assumssparams}

\subsection{Regularity Condition and Robust Strict Saddle Property.}
\label{subsec:regularitycond}

Let \\$\mathcal{O}_r := \{ R \in \R^{r \times r} : R^\top R = I\}$ be the set of $r \times r$ orthogonal matrices. We have the following definition.
\begin{definition}
Given a function $h(Z) : \R^{p \times r} \rightarrow \R$ we say that $h$ is invariant under orthogonal transformations if
\[
h(ZR) = h(Z),
\]
for all $Z \in \R^{p \times r}$ and $R \in \mathcal{O}_r$.
\end{definition}
It is easy to verify that $F$ defined in \eqref{eq:fdef} satisfies this property.

We note that the Frobenius norm of $Z$ is invariant under orthogonal transformation as well, i.e. $\|ZR\|_F = \|Z\|_F$ for all $R \in \mathcal{O}_r$. We can define the distance between two matrices $Z^1$ and $Z^2$ as follows:
\begin{equation} \label{eq:distdef}
\dist(Z^1, Z^2) := \min_{R \in \mathcal{O}_r} \, \|Z^1 - Z^2 R\|_F.
\end{equation}
For convenience, we denote by $R(Z^1,Z^2)$ the orthogonal matrix that achieves the minimum in \eqref{eq:distdef}, that is,
\begin{equation} \label{eq:rdef}
R(Z^1,Z^2) := \argmin_{R \in \mathcal{O}_r} \, \|Z^1 - Z^2 R\|_F
\end{equation}

We can now define the local regularity condition of interest in this work; these conditions were defined in a slightly more general setting in \cite{EJCandes_XLi_MSoltanolkotabi_2015,STu_RBoczar_MSimchowitz_MSoltanolkotabi_BRecht_2016}. 
%
\begin{definition} \label{def:regularitycond}
Suppose $h : \R^{p \times r}\rightarrow \R$ is invariant under orthogonal transformations. Let $Z^* \in \R^{p \times r}$ be a local minimium of $h$. Define the ball of radius $\delta$ around $Z^*$ as
\[
B(Z^*, \delta) := \left\{Z \in \R^{p \times r} :
\dist(Z,Z^*) \leq \delta \right\},
\]
where $\dist(\cdot,\cdot)$ is defined in \eqref{eq:distdef}. Then, we say that $h(Z)$ satisfies the $(\alpha, \beta, \delta)$-regularity condition at $Z^*$ 
(where $\alpha$, $\beta$, and $\delta$ are all positive quantities) if for all $Z \in B(Z^*, \delta)$, we have
\begin{equation} \label{eq:regularitydef}
\langle \nabla h(Z), Z - Z^*R\rangle \geq \alpha \, \dist(Z,Z^*)^2 + \beta \, \|\nabla h(Z)\|^2_F,  \quad \mbox{where $R = R(Z,Z^*)$.}
\end{equation}
\end{definition}

Note that $\alpha$ and $\beta$ in Definition~\ref{def:regularitycond} must satisfy $\alpha \beta \leq 1/4$ because of the Cauchy-Schwarz inequality, which indicates that for any $R \in \mathcal{O}_r$ we have
\[
\langle \nabla h(Z), Z - Z^* R\rangle
\leq \dist(Z,Z^*) \|\nabla h(Z)\|_F,
\]
and the inequality of arithmetic and geometric means,
\[
\alpha \, \dist(Z,Z^*)^2 + \beta \, \|\nabla h(Z)\|^2_F \geq 2
\sqrt{\alpha \beta} \dist(Z, Z^*) \|\nabla h(Z)\|_F.
\]
In addition, \eqref{eq:regularitydef} implies that
\begin{equation} \label{eq:gradientUB}
\beta \|\nabla h(Z)\|_F \leq \dist(Z,Z^*),
\end{equation}
holds for all $Z \in B(Z^*, \delta)$, by the Cauchy-Schwarz inequality and $\alpha \dist(Z,Z^*)^2 \geq 0$.

One important consequence of the regularity condition is local convergence of gradient descent at a linear rate.
\begin{lemma} \label{lem:localconvergence}
Let the function $h:\R^{p \times r} \to \R$ restricted to a $\delta$ neigborhood of $Z^* \in \R^{p \times r}$ satisfies the $(\alpha, \beta, \delta)$-regularity condition and suppose that $Z^0 \in B(Z^*, \delta)$. Then, after $k+1$ steps of gradient descent applied to $h$ starting from $Z^0$, with stepsizes $\nu_j \in (0, 2\beta]$ for all $j=0,1,\dotsc k$, we have
\begin{equation} \label{eq:uc7}
\dist^2(Z^{k+1}, Z^*) \leq \left[\prod_{j=0}^k (1 - 2 \nu_j \alpha) \right] \dist^2(Z^0, Z^*),
\end{equation}
so that $Z^{k+1} \in B(x^*, \delta)$.
\end{lemma}
\begin{proof}
This proof follows a similar argument to that of \cite[Lemma~7.10]{EJCandes_XLi_MSoltanolkotabi_2015} 
Denote $R(Z,Z^*)$ be defined as in \eqref{eq:rdef}. By the definition of the distance \eqref{eq:distdef}, our regularity condition \eqref{eq:regularitydef}, and $\nu_j \leq 2\beta$, we have when $\dist(Z^j,Z^*) \le \delta$ that 
\begin{align*}
&\dist^2(Z^{j+1}, Z^*) \\
&= \|Z^{j+1} - Z^*R(Z^{j+1},Z^*)\|^2_F \\
&\leq \|Z^{j+1} - Z^*R(Z^j,Z^*)\|^2_F &&
\mbox{by \eqref{eq:distdef}} \\
&= \|Z^j - \nu_j \nabla h(Z^j) - Z^*R(Z^j,Z^*)\|^2_F \\
&= \|Z^j - Z^*R(Z^j,Z^*)\|^2_F + \nu_j^2 \|\nabla h(Z^j)\|^2_F \\
&\quad -2 \nu_j \langle \nabla h(Z^j), Z^j - Z^*R(Z^j,Z^*)\rangle \\
&\leq (1-2\nu_j \alpha) \dist^2 (Z^j, Z^*) - \nu_j (2 \beta - \nu_j) \|\nabla h(Z^j)\|^2_F 
&& \mbox{by \eqref{eq:regularitydef}} \\
&\leq (1-2\nu_j \alpha) \dist^2 (Z^j, Z^*)
&& \mbox{by $\nu_j \leq 2\beta$}.
\end{align*}
Since $\alpha \beta \leq 1/4$ and $\nu_j \leq 2 \beta$, we have that $0 \leq 1 - 2 \nu_j \alpha \leq 1$. Thus $\dist(Z^{j+1},Z^*) \le \delta$ too. By applying this argument inductively for $j=0,1,\dotsc,k$, we obtain the result.
\end{proof}

We are now ready to define the robust strict saddle property, for functions invariant under orthogonal transformations.
\begin{definition} \label{def:strictsaddle}
Suppose that the twice continuously differentiable 
function $h(Z) : \R^{p \times r} \rightarrow \R$ is invariant under orthogonal transformations. 
For the positive quantities $\alpha$, $\beta$, $\gamma$, $\epsilon$, $\delta$, function $h$ satisfies the $(\alpha, \beta, \gamma, \epsilon, \delta)$-robust strict saddle property if at any point $Z$, at least one of the following applies:
\begin{enumerate}
\item There exists a local minimum $Z^* \in \R^{p \times r}$ such that   $\dist(Z,Z^*) \leq \delta$, and the function $h$ restricted to the neighborhood $\dist(Z', Z^*) \leq 2\delta$ satisfies the $(\alpha,\beta,2\delta)$-regularity condition at $Z^*$ of Definition~\ref{def:regularitycond};
\item $\lambdamin(\nabla^2 h(Z)) \leq - \gamma$; or
\item $\|\nabla h(Z)\|_F \geq \epsilon$.
\end{enumerate}
\end{definition}
Under this property, each element $Z$ of the domain belongs to at least one of three sets, each of which has a property that guarantees fast convergence of descent methods. The parameters that define these regions for low-rank matrix problems are discussed in Section~\ref{subsec:assumssparams}.

\subsection{Regularization.} \label{subsec:regularization}

Let $X^* \in \R^{n\times m}$ be a critical point of $f$
defined in \eqref{eq:fdef}, that is, $\nabla f(X^*)=0$. Suppose that
$X^*$ has rank $r \le \min(m,n)$ (see Assumption~\ref{assum:xstar}
below), and let $X^* = \Phi \Sigma \Psi^\top$ be the SVD of $X^*$,
where $\Phi \in \R^{n \times r}$ and $\Psi \in \R^{m \times r}$ have
orthonormal columns and $\Sigma$ is positive diagonal. Define
\begin{equation} \label{eq:UsVs}
U^* = \Phi \Sigma^{1/2} R, \quad V^* = \Psi \Sigma^{1/2} R
\end{equation}
for some $R \in \mathcal{O}_r$. To remove ambiguity in the matrix $W$ that corresponds to $X^*$, we add to $F(W)$ the regularization term $\rho$ defined by
\[
\rho(W) := \frac14 \left\|U^\top  U - V^\top V\right\|_F^2.
\]
The regularized optimization problem that we solve in this paper is thus
\begin{equation} \label{eq:gdef}
\underset{U \in \R^{n \times r}, V \in \R^{m \times r}}{\min} \,
G(W) := F(W) + \muhalf \rho(W).
\end{equation}
The regularization parameter $1/2$ is chosen for convenience and is sufficient to ensure the robust strict saddle property holds. Note that for $(U,V)=(U^*,V^*)$ defined in
\eqref{eq:UsVs}, and for any $R \in \cO_r$, with $W^*$ and $\hat{W}^*$
defined as in  \eqref{eq:W} and \eqref{eq:hatW}, we have
\begin{equation} \label{eq:ut1}
  (\hat{W}^*)^\top W^* = (U^*)^T U^* - (V^*)^T V^*=R^T \Sigma R - R^T
  \Sigma R = 0.
\end{equation}
We can show from \eqref{eq:UsVs} together with the definitions of $X^*$ and $W^*$ that
\begin{equation} \label{eq:XWstar}
  \|W^*\|^2 = 2 \|X^*\|, \quad \| W^* (W^*)^T \|_F = 2 \|X^*\|_F.
  \end{equation}
(We include a proof of these claims in Appendix \ref{app:matrixeqs}, for completeness.)

For the gradient of $G(W)$, we have
\begin{equation} \label{eq:ggrad}
\nabla G(W) =
\left[
\begin{matrix}
\nabla f(X)V \\
(\nabla f(X))^\top U
\end{matrix}
\right]
+ \muhalf \hat{W} \hat{W}^\top W,
\end{equation}
where $X = UV^\top$. Given matrices $D$ and $\hat{D}$ defined by
\begin{equation} \label{eq:def.D}
  D = \left[ \begin{matrix} S \\ Y \end{matrix} \right], \quad
  \hat{D} = \left[ \begin{matrix} S \\ -Y \end{matrix} \right], \quad
  \mbox{where $S \in \R^{n \times r}$ and $Y \in \R^{m \times r}$},
\end{equation}
the bilinear form of the Hessian of $G$ is given by
\begin{align}
  \nonumber
[\nabla^2 G(W)](D,D)  & =[\nabla^2 f(X)](SV^\top+UY^\top,
SV^\top+UY^\top) + 2\langle\nabla f(X), S Y^\top\rangle \\
\label{eq:ghess}
& \quad\quad + \muhalf \langle \hat{W}^\top W, \hat{D}^\top D\rangle
+ \frac{1}{4} \|\hat{W}^\top D + D^\top \hat{W}\|_F^2,
\end{align}
where $X = UV^\top$.

\subsection{Assumptions and Strict Saddle Parameters.}
\label{subsec:assumssparams}

We make the following assumptions on $f(X)$, which are identical to those found in \cite{ZZhu_QLi_GTang_MBWakin_2017}.  The first is about existence of a rank-$r$ critical point for $f$.

\begin{assumption} \label{assum:xstar}
$f(X)$ has a critical point $X^* \in \R^{n\times m}$ with rank $r$.
\end{assumption}

The second assumption is a restricted strong convexity condition for $f$.

\begin{assumption} \label{assum:frestrictedsc}
The twice continuously differentiable function $f:\R^{n \times m} \to \R$ is $(2r,4r)$-restricted strongly convex and smooth, 
that is, for any matrices $X,T \in \R^{n \times m}$ with $\rank(X) \leq 2r$ and $\rank(T) \leq 4r$, the Hessian $\nabla^2 f(X)$ satisfies
\begin{equation}
a \|T\|^2_F \leq [\nabla^2 f(X)](T,T) \leq b \|T\|^2_F,
\end{equation}
for some positive scalars $a$ and $b$.
\end{assumption}

Assumption~\ref{assum:frestrictedsc} implies that the original function, prior to splitting the variable $X$ into $U$ and $V$, obeys a form of restricted strong convexity and smoothness. This assumption is satisfied when $4r$-RIP holds, which occurs with high probability (under certain assumptions) for such problems as low-rank matrix completion and matrix sensing \cite{BRecht_MFazel_PAParrilo_2010}.

Now, we are able to define the robust strict saddle conditions for $G$. We state the following slightly abbreviated version of \cite[Theorem~1]{ZZhu_QLi_GTang_MBWakin_2017}.
\begin{theorem} \label{thm:strictsaddle}
Let $G(W)$ be defined as in \eqref{eq:gdef}. 
For the critical point $X^* \in \R^{n \times m}$, with rank $r$, suppose that $X^* = U^* (V^*)^T$, where $U^* \in \R^{n \times r}$ and $V^* \in \R^{m \times r}$ are defined as in (\ref{eq:UsVs}), and define $W^* = \left[ \begin{matrix} U^* \\ V^* \end{matrix} \right]$. 
Let $\dist(\cdot,\cdot)$ be defined as in \eqref{eq:distdef}, and let $\sigma_r(Z)>0$ denote the $r$-th singular value of the matrix $Z$. 
Suppose that Assumptions~\ref{assum:xstar} and \ref{assum:frestrictedsc} are satisfied for positive $a$ and $b$ such that 
\[
    \frac{b-a}{a+b} \le \frac{1}{100} \frac{\sigma_r^{3/2}(X^*)}{\|X^*\|_F\|X^*\|^{1/2}}.
\]
Define the following regions of the space of matrices $\R^{(m+n) \times r}$:
\begin{align*}
\cR_1 &:= \left\{W : \dist(W, W^*) \leq \sigma_r^{1/2}(X^*)\right\}, \\
\cR_2 &:= \Big\{W : \sigma_r(W) \leq \sqrt{\frac12}\sigma_r^{1/2}(X^*), \;\; \|WW^\top\|_F \leq \frac{20}{19} \|W^* (W^*)^\top\|_F \Big\}, \\
\cR'_3 &:= \Big\{W : \dist(W,W^*) > \sigma_r^{1/2}(X^*), \;\; \|W\| \leq \frac{20}{19} \|W^*\|, \\
&\quad\quad\quad \sigma_r(W) > \sqrt{\frac12} \sigma_r^{1/2}(X^*), \;\; \|WW^\top\|_F \leq \frac{20}{19}\|W^*(W^*)^\top\|_F \Big\}, \\
\cR''_3 &:= \Big\{W : \|W\| > \frac{20}{19} \|W^*\| = \sqrt{2}\frac{20}{19} \|X^*\|^{1/2}, 
\;\; \|WW^\top\|_F \leq \frac{10}{9}\|W^*(W^*)^\top\|_F \Big\}, \\
\cR'''_3 &:= \Big\{W : \|WW^\top\|_F > \frac{10}{9}\|W^*(W^*)^\top\|_F = \frac{20}{9} \|X^*\|_F \Big\}.
\end{align*}
(Note that the definitions of $\cR_3''$ and $\cR_3'''$ make use of \eqref{eq:XWstar}.)
Then there exist positive constants $\calpha$, $\cbeta$, $\cgamma$, and $\cepsilon$ such that $G(W)$ has the following strict saddle property.
\begin{enumerate}
\item For any $W \in \cR_1$, $G(W)$ satisfies the local regularity condition:
\begin{align} \label{eq:gregular}
\langle \nabla G(W), & W - W^* R(W,W^*)\rangle \\
\nonumber
& \geq \calpha \sigma_r(X^*) \, \dist^2 (W,W^*)
+ \frac{\cbeta}{\|X^*\|} \|\nabla G(W)\|_F^2,
\end{align}
where $\dist(W,W^*)$ is defined in \eqref{eq:distdef} and $R(W,W^*)$ is defined in \eqref{eq:rdef}. That is, definition \eqref{eq:regularitydef} is satisfied with $h=G$, $x=W$, $\alpha= \calpha \sigma_r(X^*)$, $\beta = \cbeta \|X^*\|^{-1}$, and $\delta=\sigma_r^{1/2}(X^*)$.

\item For any $W \in \cR_2$, $G(W)$ has a direction of large negative curvature, that is, 
\begin{equation} \label{eq:gnegcurve}
\lambdamin(\nabla^2 G(W)) \leq -\cgamma \sigma_r(X^*).
\end{equation}
\item For any $W \in \cR_3 = \cR'_3 \cup \cR''_3 \cup \cR'''_3$,
  $G(W)$ has a large gradient, that is, 
  \begin{subequations} \label{eq:glargegrad}
    \begin{alignat}{2}
      \|\nabla G(W)\|_F &\geq \cepsilon \sigma_r^{3/2} (X^*), \quad &&\mbox{for all $W \in \cR'_3$;} \label{eq:glargegrad1} \\
\|\nabla G(W)\|_F &\geq \cepsilon \|W\|^3, \quad  && \mbox{for all $W \in  \cR''_3$}; \label{eq:glargegrad2} \\
\|\nabla G(W)\|_F &\geq
\cepsilon \|WW^\top\|_F^{3/2}, \quad && \mbox{for all $W \in \cR'''_3$.} \label{eq:glargegrad3}
\end{alignat}
\end{subequations}
\end{enumerate}
\end{theorem}

It follows from this theorem that the function $G$ satisfies the robust strict saddle property of Definition~\ref{def:strictsaddle} with
\[
\alpha = \calpha \sigma_r(X^*),
\quad \gamma = \cgamma \sigma_r(X^*), \quad
\delta = \sigma_r^{1/2}(X^*), \quad
\beta = \cbeta \|X^*\|^{-1},
\]
and different values of $\epsilon$ that depend on the region:
\begin{align*}
  \epsilon_{R_3'} &= \cepsilon \sigma_r(X^*)^{3/2}, \\
  \epsilon_{R_3''} &=
\cepsilon \|W\|^3 \geq \cepsilon \left(\sqrt{2}\frac{20}{19}\right)^3 \|X^*\|^{3/2}, \\
\epsilon_{R_3'''} &= \cepsilon \|WW^\top\|_F^{3/2} \geq
\cepsilon \left(\frac{20}{19}\right)^{3/2} \|X^*\|_F^{3/2}.
\end{align*}

The regions defined in Theorem~\ref{thm:strictsaddle} span the space of matrices occupied by $W$ but are not a partition, that is,
\[
\cR_1 \cup \cR_2 \cup \cR_3' \cup \cR_3'' \cup \cR_3''' = \cR^{(n+m)\times r}.
\]
The constants $\calpha$, $\cbeta$, $\cgamma$, and $\cepsilon$ in this theorem may vary between problems in this class. Settings that work for all classes mentioned are
\[
\calpha = \frac{1}{16},
\quad \cbeta = \frac{1}{260},
\quad \cgamma = \frac16,
\quad \cepsilon = \frac{1}{50}.
\]
(For clarity, we use the same constant, $\cepsilon$, for each equation in (\ref{eq:glargegrad}) even though slightly tighter bounds are possible if each is treated individually.)
Note that these constants are used in the algorithm presented below. 
For convenience, we define the following combination of the parameters above, which is used repeatedly in the algorithms and analysis below:
\begin{equation}
    \label{eq:cfac}
    \cfac := \cepsilon \mincacb.
\end{equation}

In addition to the strict saddle assumption, we make the following standard assumptions on $G$, concerning compactness of the level set defined by the initial point $W^0$ and smoothness.
\begin{assumption} \label{assum:compactlevelset}
  Given an initial iterate $W^0$, the level set defined by $\mathcal{L}_G(W^0) = \{W | G(W) \le G(W^0)\}$ is compact.
\end{assumption}

\begin{assumption} \label{assum:GC22}
The function $G$ is twice Lipschitz continuously differentiable with respect to the Frobenius norm on an open neighborhood of $\mathcal{L}_G(W^0)$, and we denote by $L_g$ and $L_H$ the respective Lipschitz constants for $\nabla G$ and $\nabla^2 G$ on this set.
\end{assumption}


Under Assumptions~\ref{assum:compactlevelset} and \ref{assum:GC22},
there exist scalars $\Glow$, $U_g > 0$, $U_H > 0$, and
$R_{\mathcal{L}} > 0$ such that the following are satisfied for all
$W$ in an open neighborhood of $\mathcal{L}_G(W^0)$:
\begin{equation}\label{eq:boundscompact}
	G(W) \geq \Glow,\quad \|\nabla G(W)\|_F \leq U_g,
	\quad \|\nabla^2 G(W)\| \leq U_H, \quad \|W\| \leq R_{\mathcal{L}},
\end{equation}
where the third condition is taken on the ``unrolled" Hessian of $G$.
These assumptions also imply the following well known inequalities, for $W$ and $D$ such that all points in the convex hull of $W$ and $D$ lie in the neighborhood of the level set mentioned above:
\begin{subequations} \label{eq:Lgh}
\begin{align}
  \label{eq:Lg}
G(W+D) & \le G(W) + \langle \nabla G(W), D \rangle + \frac{L_g}{2} \|
D\|_F^2, \\
\label{eq:LH}
G(W+D) & \le G(W) + \langle \nabla G(W), D \rangle + \frac12 \langle D,
\nabla^2 G(W) D\rangle + \frac{L_H}{6} \| D\|_F^3.
\end{align}
\end{subequations}

Finally, we make an assumption about knowledge of the Lipschitz constant of the gradient of $f(X)$.
\begin{assumption} \label{assum:lipknown}
The gradient $\nabla f(X)$ is Lipschitz continuous on an open neighborhood of
\[
\left\{Z : Z = UV^\top, \left[\begin{matrix} U\\V\end{matrix}\right]
\in \mathcal{L}_G(W^0) \right\},
\]
and the associated constant, denoted by $L_{\nabla f}$, is known or can be efficiently estimated. 
That is, for  any $X_a$, $X_b$ in the set defined above, we have
\begin{equation} \label{eq:flipgrad}
\|\nabla f(X_a) - \nabla f(X_b)\|_F \leq L_{\nabla f} \|X_a - X_b\|_F.
\end{equation}
\end{assumption}

In many interesting applications, $L_{\nabla f}$ is easily discerned or can be efficiently computed. 
An example is the low-rank matrix completion problem where a set of observations $M_{ij}$, $(i,j) \in \Omega$ is made of a matrix, and the objective is $f(X) = \frac12 \sum_{(i,j) \in \Omega} \, (X_{ij}-M_{ij})^2$. Here, we have  $L_{\nabla f} = 1$. 
A similar example is matrix sensing problem, in which $f(X) = \frac12 \|\mathcal{A}(X) - y \|^2_2$, where $\mathcal{A} : \R^{n \times m} \rightarrow \R^p$ is a known linear measurement operator and $y \in \R^p$ is the set of observations. 
In this case, we can write $\mathcal{A}(X) = [\langle A_i, X\rangle]_{i=1,2,\dotsc,p}$ where
$A_i \in \R^{n \times m}$ for all $i = 1,2,\dotsc, p$, so that $\nabla f(X) = \sum_{i=1}^p (\langle A_i, X \rangle -y_i) A_i$ and thus $L_{\nabla f} = \sum_{i=1}^p  \|A_i\|^2_F$.

\section{The Algorithm.} \label{sec:alg}

We describe our algorithm in this section. 
Sections~\ref{sec:algdescription} and \ref{subsec:meo} give a detailed description of each element of the algorithm, along with a description of how the key parameters in the definition of strict saddle are estimated.
Section~\ref{sec:algcorrectness} shows that the algorithm properly identifies the strict saddle regions once the parameter $\gamma_k$ is a sufficiently good estimate of $\sigma_r(X^*)$.

\begin{algorithm}[ht!]
\caption{Line-Search Algorithm For Strict Saddle Functions}
\label{alg:stupid}
\begin{algorithmic}
\STATE \emph{Inputs:} Optimality tolerances  $\epsg \in (0, 1)$,
$\epsH \in (0, 1)$; starting point $W^0$;
starting guess $\gamma_0  \geq \sigma_r(X^*) > 0$;
step acceptance parameter $\eta \in (0,1)$;
backtracking parameter $\theta \in (0,1)$;
Lipschitz constant $L_{\nabla f} \geq 0$;
\STATE \emph{Optional Inputs:} Scalar $M > 0$ such that $\|\nabla^2 G(W)\| \leq M$
for all $W$ in an open neighborhood of $\mathcal{L}_G(W^0)$;
\STATE converged $\leftarrow$ False;
\FOR{$k=0,1,2,\dotsc$}
\IF[Large Gradient, Take Steepest Descent Step]{$\|\nabla G(W^k)\|_F \geq \cfac \gamma_k^{3/2}$} 
\STATE Compute $\nu_k=\theta^{j_k}$, 
where $j_k$ is the smallest nonnegative integer such that \vspace{-0.25cm}
\begin{equation} \label{eq:lsdecreasegrad}
G(W^k - \nu_k \nabla G(W^k)) < G(W^k) - \eta \nu_k \|\nabla G(W^k)\|_F^2;
\end{equation} \vspace{-0.5cm}
\STATE $W^{k+1} \leftarrow W^k - \nu_k \nabla G(W^k)$;
\STATE $\gamma_{k+1} \leftarrow \gamma_k$;
\ELSE[Seek Direction of Large Negative Curvature]
\STATE Call Procedure~\ref{alg:meo} with $H = \nabla^2 G(W^k)$,
$\epsilon = \cgamma \gamma_k$, and $M$ (if provided);
\IF[Initialize Local Phase]{Procedure~\ref{alg:meo} certifies $\lambdamin(\nabla^2 G(W^k)) \geq -\cgamma \gamma_k$}
\STATE $\alpha_k \leftarrow \calpha \gamma_k$;
\STATE $\delta_k \leftarrow \sqrt{2} \gamma_k^{1/2}$;
\STATE $\beta_k \leftarrow \frac{2 \cbeta}{(\delta_k + \|W^k\|_F)^2}$;
\IF[Try Local Phase]{$\alpha_k \beta_k \leq \frac14$ \textbf{and} $\|\nabla G(W^k)\|_F \leq \frac{\delta_k}{\beta_k}$ \textbf{and}
\\ \quad$2\|\nabla f(X^k)\|_F + \muhalf \|(\hat{W}^k)^\top W^k\|_F
\leq (2L_{\nabla f} + \muhalf)(2\|W^k\|_F + \delta_k) \delta_k$}  
\STATE Call Algorithm~\ref{alg:local} with $W^k_0 = W^k$, $\epsg$, $\epsH$,
$\alpha_k$, $\beta_k$, $\delta_k$, $\eta$, $\theta$, and $L_{\nabla f}$ to obtain outputs $W^{k+1}$, $T_k$, converged;
\IF[Local Phase Found Near-Optimal Point]{converged $=$ True}
\STATE Terminate and return $W^{k+1}$;
\ENDIF
\ELSE[Do not update $W$]
\STATE $W^{k+1} \leftarrow  W^k$;
\ENDIF
\STATE $\gamma_{k+1} \leftarrow \tfrac12 \gamma_k$;
\ELSE[Search in Large Negative Curvature Direction $s$ from Procedure~\ref{alg:meo}]
\STATE Set $D^k \leftarrow -\mathrm{sgn}\left(\langle S, \nabla G(W^k)\rangle\right) |\langle S, \nabla^2 G(W^k) S\rangle | S$,
where $S$ is the $\R^{(n+m)\times r}$ matrix formed by reshaping the output vector $s$ from Procedure~\ref{alg:meo};
\STATE Compute $\nu_k=\theta^{j_k}$, 
where $j_k$ is the smallest nonnegative integer such that \vspace{-0.25cm}
\begin{equation} \label{eq:lsdecreasenc}
G(W^k + \nu_k D^k) < G(W^k) +
\eta \frac{\nu_k^2}{2} \langle D^k, \nabla^2 G(W^k) D^k \rangle;
\end{equation} \vspace{-0.5cm}
\STATE $W^{k+1} \leftarrow W^k + \nu_k D^k$;
\STATE $\gamma_{k+1} \leftarrow \gamma_k$;
\ENDIF
\ENDIF
\ENDFOR
\end{algorithmic}
\end{algorithm}

\begin{algorithm}
\begin{algorithmic}
\caption{Local Phase for Strict Saddle Problems}
\label{alg:local}
\STATE \emph{Inputs:} Optimality tolerances $\epsg \in (0,1]$, $\epsH \in (0,1]$;
    starting point $W^k_0$; 
    strict saddle parameters $\alpha_k$, $\beta_k$, $\delta_k > 0$;
    step acceptance parameter $\eta \in (0, 1)$;
    backtracking parameter $\theta \in (0,1)$;
    Lipschitz constant $L_{\nabla f} \geq 0$;
\STATE \emph{Outputs:} $W^{k+1}$, $T_k$, converged;
\STATE converged $\leftarrow$ False;
\STATE $\kappa_0 \leftarrow 1$;
\STATE $\tau_0 \leftarrow (2L_{\nabla f} + \muhalf)(2\|W^k_0\|_F + \delta_k) \delta_k$;
\STATE $t \leftarrow 0$;
\WHILE{$\|\nabla G(W^k_t)\|_F \leq \frac{\sqrt{\kappa_t}}{\beta_k}\delta_k$
\textbf{and} $2\|\nabla f(X^k_t)\|_F + \muhalf \|(\hat{W}^k_t)^\top W^k_t\|_F \leq \tau_t$}
\STATE Compute $\nu_t=2\beta_k\theta^{j_t}$, 
where $j_t$ is the smallest nonnegative integer such that \vspace{-0.25cm}
\begin{equation} \label{eq:lsdecreasegradlocal}
G(W^k_t - \nu_t \nabla G(W^k_t)) < G(W^k_t) - \eta \nu_t \|\nabla G(W^k_t)\|_F^2;
\end{equation} \vspace{-0.5cm}
\STATE $W^{k}_{t+1} \leftarrow W^k_t - \nu_t \nabla G(W^k_t)$;
\STATE $\kappa_{t+1} \leftarrow (1-2\nu_t \alpha_k)\kappa_t$;
\STATE $\tau_{t+1} \leftarrow (2L_{\nabla f} + \muhalf)(2\|W^k_{t+1} \|_F + \sqrt{\kappa_{t+1}} \delta_k)\sqrt{\kappa_{t+1}}\delta_k$;
\STATE $t \leftarrow t + 1$;
\IF{$\|\nabla G(W^k_t)\| \leq \epsg$ \textbf{and}
$2\|\nabla f(X^k_t)\|_F + \muhalf \|(\hat{W}^k_t)^\top W^k_t\|_F \leq \epsH$}
\STATE converged $\leftarrow$ True;
\STATE \textbf{break}
\ENDIF
\ENDWHILE
\STATE $W^{k+1} \leftarrow W^k_t$;
\STATE $T_k \leftarrow t$;
\end{algorithmic}
\end{algorithm}

\subsection{Line-Search Algorithm for Strict Saddle Functions.}
\label{sec:algdescription}

Our main algorithm is defined in Algorithm~\ref{alg:stupid}. 
At each iteration, it attempts to identify the region curently occupied by $W^k$. 
A step appropriate to the region is computed.
The critical parameter in identifying the regions is $\gamma_k$, which is our upper estimate of the parameter $\gamma = \sigma_r(X^*)$ (ignoring the constant $\cgamma$), which plays a key role in the definitions of the regions in Theorem~\ref{thm:strictsaddle}. 
When the large gradient condition is satisfied (that is, when $W^k$ is estimated to lie in $\cR_2$), a gradient descent step is taken. 
Similarly, when the condition for large negative curvature is satisfied (that is, when $W^k$ is estimated to lie in $\cR_1$), a direction of significant negative curvature is found by using Procedure~\ref{alg:meo}, and a step is taken in a scaled version of this direction. 
(The approach for negative curvature steps is similar to that of \cite{CWRoyer_MONeill_SJWright_2019}.) 
In both cases, a backtracking line search is used to ensure sufficient decrease in $G$.

When neither of these two scenarios are satisfied, the algorithm enters a ``local phase'' defined by Algorithm~\ref{alg:local}. 
This process begins by estimating a number of the robust strict saddle parameters of Definition~\ref{def:strictsaddle}. 
These parameters are chosen so that the local phase will converge linearly to an approximate second-order point {\em provided that} $\gamma_k$ is within a factor of two of the value of $\sigma_r(X^*)$, that is,
\begin{equation} \label{eq:gammastardef}
  \gamma_k \in \Gamma(X^*) := \left[\tfrac12 \sigma_r(X^*), \sigma_r(X^*)\right).
\end{equation}
When $\gamma_k$ is in the interval $\Gamma(X^*)$, the value of $\delta_k$ defined in Algorithm~\ref{alg:stupid} is an upper bound on $\delta$, while $\alpha_k$ and $\beta_k$ are lower bounds on $\alpha$ and $\beta$ from Definition~\ref{def:strictsaddle}. 
Conditions are checked during the execution of the local phase to monitor for fast convergence.
These conditions will be satisfied whenever $\gamma_k \in \Gamma(X^*)$.
If the conditions are not satisfied, then \eqref{eq:gammastardef} does not hold, so we halve the value of $\gamma_k$ and proceed without taking a step in $W$.

The local phase, Algorithm~\ref{alg:local}, begins by initializing an inner iteration counter $t$ as well as the scalar quantities $\kappa_t$ and $\tau_t$, which are used to check for linear convergence of $W^k_t$, for $t=0,1,2,\dotsc$ to a point satisfying \ref{eq:optconds}. 
Each iteration of the local phase consists of a gradient descent step with a line search parameter $\nu_t$ obtained by backtracking from an initial value of $2 \beta_k$.
Once a stepsize $\nu_t$ is identified and the gradient descent step is taken, $\kappa_t$ is updated to reflect the linear convergence rate that occurs when $\gamma_k \in \Gamma(X^*)$ and $W^k \in \cR_1$. 
At each iteration, this linear convergence rate is checked, by looking at the gradient of $G(W)$ as well as the the gradient of the original function $f(X)$.
Under the assumptions discussed in Section \ref{subsec:assumssparams}, these quantities provide estimates for \eqref{eq:optconds}, since the minimum eigenvalue of the Hessian of $G(W)$ can be lower bounded using $\nabla f(X)$ (see Section~\ref{subsec:lb} for details). 
These checks ensure that the local phase either converges at a linear rate to a point satisfying \eqref{eq:optconds}, or else exits quickly with a flag indicating that the current estimate $\gamma_k$ of $\sigma_r(X^*)$ is too large.

\subsection{Minimum Eigenvalue Oracle.} \label{subsec:meo}

\floatname{algorithm}{Procedure}

\begin{algorithm}[ht!]
\caption{Minimum Eigenvalue Oracle}
\label{alg:meo}
\begin{algorithmic}
  \STATE \emph{Inputs:} Symmetric matrix $H \in \R^{N \times N}$,
  tolerance $\epsilon>0$;
  \STATE \emph{Optional input:} Scalar $M>0$ such that $\| H \| \le M$;
  \STATE \emph{Outputs:} An estimate $\lambda$ of $\lambdamin(H)$ such that $\lambda \le - \epsilon/2$ and vector $s$ with $\|s\|=1$ such that $s^\top H s =\lambda$ OR a certificate that $\lambdamin(H) \ge -\epsilon$. In the latter case,  the certificate is false with probability $\rho$ for some $\rho \in [0,1)$.
\end{algorithmic}
\end{algorithm}

The Minimum Eigenvalue Oracle (Procedure~\ref{alg:meo}) is called when the large gradient condition $\|\nabla G(W^k)\|_F \geq \cfac \gamma_k^{3/2}$ does not hold. 
The input matrix $H$ is the ``unrolled'' Hessian of $G(W)$, a symmetric matrix of dimension $N=(n+m)r$.  
The oracle either returns a direction along which the Hessian has curvature at most $-\epsilon/2$, or certifies that the minimum curvature is greater than $-\epsilon$. 
In the latter case, the certificate may be wrong with some probability $\rho \in (0,1)$, where $\rho$ is a user-specified parameter. 
When the certificate is returned, Algorithm~\ref{alg:stupid} enters the local phase.

Procedure~\ref{alg:meo} can be implemented via any method that finds the smallest eigenvalue of $H$ to an absolute precision of $\epsilon/2$ with probability at least $1-\rho$. 
(A deterministic implementation based on a full eigenvalue decomposition would have $\rho=0$.)
Several possibilities for implementing Procedure~\ref{alg:meo} have been proposed in the literature, with various guarantees.
In our setting, in which Hessian-vector products and vector operations are the fundamental operations, Procedure~\ref{alg:meo} can be implemented using the Lanczos method with a random starting vector (see \cite{YCarmon_JCDuchi_OHinder_ASidford_2018}). This approach does not require explicit knowledge of $H$, only the ability to find matrix-vector products of $H$ with a given vector. The following result from \cite[Lemma~2]{CWRoyer_MONeill_SJWright_2019} verifies the effectiveness of this approach.
\begin{lemma} \label{lem:randlanczos}
Suppose that the Lanczos method is used to estimate the smallest eigenvalue of $H$ starting with a random vector uniformly generated on the unit sphere, where $\|H\| \le M$. For any $\rho \in [0,1)$, this approach finds the smallest eigenvalue of $H$ to an absolute precision of $\eps/2$, together with a corresponding direction $s$, in at most
\begin{equation} \label{eq:rlanc}
	\min \left\{N, 1+\ceil*{\frac12 \ln (2.75 N/\rho^2)
	\sqrt{\frac{M}{\epsilon}}} \right\} \quad \mbox{iterations},
\end{equation}	
with probability at least $1-\rho$.
\end{lemma}

Procedure~\ref{alg:meo} can be implemented by outputting the approximate eigenvalue $\lambda$ for $H$, determined by the randomized Lanczos process, along with the corresponding direction $s$, provided that $\lambda \le -\eps/2$. 
When $\lambda>-\eps/2$, Procedure~\ref{alg:meo} returns the certificate that $\lambdamin(H) \ge -\eps$, which is correct with probability at least $1-\rho$.

We note here that while the second-order optimality conditions could be checked using the minimum eigenvalue oracle inside of the local phase, this procedure can  be quite inefficient compared to the rest of the algorithm. 
From the result of Lemma \ref{lem:randlanczos}, it is clear that attempting to verify that $\lambdamin(\nabla^2 G(W)) \geq -\epsH$ holds could require as many as $\min\left\{(n+m)r,\epsH^{-1/2}\right\}$ gradient evaluations/Hessian-vector products. 
This dependence on $\epsH$ --- worse than the logarithmic dependence on tolerances that is the stated goal of this work. 
We avoid this issue by using $\nabla f(X)$ to estimate a lower bound of the spectrum of $\nabla^2 G(W)$, as discussed in the following section. 
This allows us to maintain the logarithmic dependence on our optimization tolerances $\epsg$ and $\epsH$ while still ensuring convergence to an approximate second-order point.

\subsection{Lower-Bounding the Spectrum of $\nabla^2 G(W^k)$.} 

\label{subsec:lb}

We now prove two technical results about quantities that lower-bound the minimum eigenvalue of Hessian of $G(W)$.
These bounds motivate some unusual expressions in Algorithms~\ref{alg:stupid} and \ref{alg:local} that allow us to  check the second-order approximate optimality condition in \eqref{eq:optconds} {\em indirectly}, and ultimately at lower cost than a direct check of $\lambdamin(\nabla^2 G(W))$.


\begin{lemma} \label{lem:nclowerbound}
  Suppose that Assumption~\ref{assum:frestrictedsc} holds, and let $W$ and $\hat{W}$ be defined in \eqref{eq:W} and \eqref{eq:hatW}, respectively, with $G(W)$ defined in \eqref{eq:gdef}. Then we have
\[
\lambdamin(\nabla^2 G(W)) \geq -2\|\nabla f(X)\|_F
- \muhalf \|\hat{W}^\top W\|_F,
\]
where $X = UV^\top$.
\end{lemma}
\begin{proof}
Let $D$ be defined as in \eqref{eq:def.D}, with component matrices $S$ and $Y$. 
Since $\rank(X) = \rank(UV^\top) \leq r$ and $\rank(SV^\top+UY^\top) \leq 2r$, we have by Assumption~\ref{assum:frestrictedsc} that
\[
\langle SV^\top+UY^\top, \nabla^2 f(X) (SV^\top +
UY^\top)\rangle \geq a \|D\|^2_F \geq 0.
\]
It follows from \eqref{eq:ghess} and the Cauchy-Schwarz inequality that
\begin{align*}
\langle D, \nabla^2 G(W) D \rangle &= \langle SV^\top+UY^\top,
\nabla^2 f(X) (SV^\top + UY^\top)\rangle + 2\langle\nabla f(X), S Y^\top\rangle \\
& \quad\quad + \muhalf \langle \hat{W}^\top W, \hat{D} D\rangle
+ \frac{1}{4} \|\hat{W}^\top D + D^\top \hat{W}\|_F^2 \\
&\geq -2 \|\nabla f(X)\|_F \|S\|_F \|Y\|_F
- \muhalf \|\hat{W}^\top W\|_F \|\hat{D}\|_F \|D\|_F.
\end{align*}
Defining $D'$ and $\hat{D}'$ by
\[
D' = \left[ \begin{matrix} S' \\ Y'  \end{matrix} \right] \in
\arg\min_D \frac{\langle D, \nabla^2 G(W) D\rangle}{\|D\|^2_F}, \quad
\hat{D}' = \left[ \begin{matrix} S' \\ -Y'  \end{matrix} \right]
\]
we have
\begin{align*}
\lambdamin(\nabla^2 G(W))
&\geq -2 \frac{\|\nabla f(X)\|_F \|S'\|_F \|Y'\|_F}{\|D'\|^2_F}
- \muhalf \frac{\|\hat{W}^\top W\|_F \|\hat{D}'\|_F \|D'\|_F}{\|D'\|^2_F} \\
&\geq -2 \|\nabla f(X)\|_F -\muhalf \|\hat{W}^\top W\|_F,
\end{align*}
where the final inequality follows by $\|S'\|_F \leq \|D'\|_F$, $\|Y'\|_F \leq \|D'\|_F$, and $\|\hat{D}\|_F = \|\hat{D}'\|_F$.
\end{proof}

Next, we show how to relate the lower bound of Lemma~\ref{lem:nclowerbound} to the distance between $W$ and $W^*$. The following result has an expression that is similar to one that appears in Algorithm~\ref{alg:stupid}, in the condition that determines whether to call Algorithm~\ref{alg:local}.

\begin{lemma} \label{lem:ncbound}
Suppose that Assumptions~\ref{assum:xstar}, \ref{assum:frestrictedsc}, and \ref{assum:lipknown} hold. Let $W$ and $\hat{W}$ be as defined as in \eqref{eq:W} and \eqref{eq:hatW}, respectively, and  let $X^*$ be as in Assumption~\ref{assum:xstar}, with $U^*$ and $V^*$ (and hence $W^*$) defined as in \eqref{eq:UsVs}, for some $R \in \mathcal{O}_r$. Then we have
\[
2\|\nabla f(X)\|_F + \muhalf \|\hat{W}^\top W\|_F
\leq \left( 2L_{\nabla f} + \muhalf \right) \left( 2 \|W\|_F + \dist(W,W^*) \right) \dist(W,W^*).
\]
\end{lemma}
\begin{proof}
We begin by bounding $\|\nabla f(X)\|_F$.  Given $W$, and hence $U$ and $V$, let $R=R(W,W^*)$ in \eqref{eq:UsVs} be the matrix in $\mathcal{O}_r$ that minimizes $\|WR-W^*\|_F$. (Note that the same minimizes $\|\hat{W}R-\hat{W}^*\|_F$, that is, $R(W,W^*) = R(\hat{W},\hat{W}^*)$). By Assumption~\ref{assum:lipknown} and the definition of $X^*$, we have
\begin{align*}
\|\nabla f(X)\|_F &= \|\nabla f(X) - \nabla f(X^*)\|_F \\
&= \|\nabla f(UR (VR)^\top) - \nabla f(U^* (V^*)^\top)\|_F \\
&\leq L_{\nabla f} \|UR(VR)^\top - U^*(V^*)^\top\|_F.
\end{align*}
Further, we have
\begin{align*}
\|UR(VR)^\top - U^*(V^*)^\top\|_F
&=  \|UR(VR)^\top - UR(V^*)^\top
+ UR(V^*)^\top - U^*(V^*)^\top\|_F \\
&\leq \|UR(R^\top V^\top - (V^*)^\top)\|_F
+ \|(UR - U^*)(V^*)^\top\|_F \\
&\leq \|UR\|_F \|R^\top V^\top - (V^*)^\top\|_F
+ \|(V^*)^\top\|_F \|UR-U^*\|_F \\
&\leq (\|W\|_F + \|W^*\|_F) \dist(W,W^*),
\end{align*}
so that
\begin{equation} \label{eq:ak1}
2 \|\nabla f(X)\|_F \leq 2 L_{\nabla f}
(\|W\|_F + \|W^*\|_F) \dist(W,W^*).
\end{equation}

To bound  $\|\hat{W}^\top W\|_F$, we have by $(\hat{W}^*)^\top W^* = 0$ that 
\begin{align}
  \nonumber
\|\hat{W}^\top W\|_F &= \|(\hat{W}R)^\top WR \|_F \\
  \nonumber
&= \|(\hat{W}R)^\top WR - (\hat{W}R)^\top W^*
+ (\hat{W}R)^\top W^* - (\hat{W}^*)^\top W^*\|_F \\
  \nonumber
&\le \|(\hat{W}R)^\top (WR - W^*)\|_F +
\|((\hat{W}R)^\top- (\hat{W}^*)^\top) W^*\|_F \\
  \nonumber
  &\leq \|(\hat{W}R)^\top\|_F \|WR - W^*\|_F + \|W^*\|_F \|(\hat{W}R)^\top - (\hat{W}^*)^\top\|_F \\
  \label{eq:ak2}
&\leq (\|W\|_F + \|W^*\|_F) \dist(W,W^*).
\end{align}

By combining \eqref{eq:ak1} and \eqref{eq:ak2}, we have
\[
2\|\nabla f(X)\|_F + \muhalf \|\hat{W}^\top W\|_F
\leq (2L_{\nabla f} + \muhalf ) (\|W\|_F + \|W^*\|_F) \dist(W,W^*).
\]
To obtain the result, note that for $R=R(W,W^*)$, we have
\[
\|W^*\|_F = \|W^* - WR + WR\|_F \leq \|W^* - WR\|_F + \|WR\|_F = \dist(W,W^*) + \|W\|_F.
\]
\end{proof}

\subsection{Behavior of Algorithm~\ref{alg:stupid} under Accurate Parameter Estimates.}
\label{sec:algcorrectness}

In this section, we will show that Algorithm~\ref{alg:stupid} correctly identifies the regions $\cR_2$ and $\cR_3$ when $\gamma_k$ lies in the interval $\Gamma(X^*)$ defined by \eqref{eq:gammastardef}, that is, it is within a factor of two of $\sigma_r(X^*)$.
Additionally, once the local phase is reached with $\gamma_k \in \Gamma(X^*)$ and $W^k \in \cR_1$, the sequence $W^k_t$, $t=0,1,2,\dotsc$ generated in the local phase converges at a linear rate to a point satisfying (\ref{eq:optconds}).

These results are crucial building blocks for the main convergence results, as they show that once the parameter $\gamma_k$ is a good estimate of  $\sigma_r(X^*)$, the algorithm behaves well enough (with high probability) to not reduce $\gamma_k$ any further, and converges rapidly thereafter at a rate that depends mostly on a polynomial in the inverse of $\sigma_r(X^*)$ rather than of the tolerances in \eqref{eq:optconds}.

We begin by showing that $\cR_2$ is properly identified with high probability, in the sense that the algorithm is likely to take a successful negative curvature step when $W^k \in \cR_2$.


\begin{lemma} \label{lem:r2step}
Let Assumptions~\ref{assum:xstar} and \ref{assum:frestrictedsc} hold. 
Suppose that $W^k \in \cR_2$ and that $\gamma_k \in \Gamma(X^*)$.  
Then a negative curvature step will be taken with probability at least $1 - \rho$ by Algorithm~\ref{alg:stupid} whenever $\|\nabla G(W^k)\|_F < \cfac \gamma_k^{3/2}$. 
\end{lemma}
\begin{proof}
Since $\gamma_k \in \Gamma(X^*)$, we have $\gamma_k \leq \sigma_r(X^*)$. By Part 2 of Theorem~\ref{thm:strictsaddle}, we have
\[
\lambdamin (\nabla^2 G(W^k)) \leq -\cgamma \sigma_r(X^*) \leq -\cgamma \gamma_k,
\]
so a negative curvature step is taken whenever Procedure~\ref{alg:meo} is invoked (which occurs when $\|\nabla G(W^k)\|_F < \cfac \gamma_k^{3/2}$), provided that Procedure~\ref{alg:meo} finds a direction of negative curvature, which happens with probability at least $1-\rho$.
\end{proof}

A similar result holds for $\cR_3$.

\begin{lemma} \label{lem:r3step}
Let Assumptions \ref{assum:xstar} and \ref{assum:frestrictedsc} hold. Suppose that $W^k \in \cR_3$ and that $\gamma_k \in \Gamma(X^*)$.  
Then $\|\nabla G(W^k)\|_F \geq \cfac \gamma_k^{3/2}$, so a large-gradient step is taken by Algorithm~\ref{alg:stupid}.
\end{lemma}
\begin{proof}
First, assume that $W^k \in \cR'_3$. Then, by \eqref{eq:glargegrad1} in Theorem~\ref{thm:strictsaddle}, we have
\[
\|\nabla G(W^k)\|_F \geq \cepsilon \sigma_r(X^*)^{3/2} \geq \cfac \gamma_k^{3/2},
\]
where the second inequality follows from $\gamma_k \in \Gamma(X^*)$. 
Thus, a gradient step will be taken if $W^k \in \cR'_3$.

Next, assume that $W^k \in \cR''_3$. Since $\gamma_k \in \Gamma(X^*)$, we have that $\alpha = \calpha \sigma_r(X^*) \geq \calpha \gamma_k$. 
By Part 1 of Theorem~\ref{thm:strictsaddle}, the $(\alpha, \beta, \delta)$-regularity condition of Definition~\ref{def:regularitycond} holds with $\alpha = \calpha \sigma_r(X^*)$ and $\beta = \cbeta \|X^* \|^{-1}$. 
Since $\alpha \beta \le 1/4$ (as noted following Definition~\ref{def:regularitycond}), we have
\[
\frac14 \geq \alpha \beta= \frac{\calpha \cbeta \sigma_r(X^*)}{ \| X^*\|} \ge
\frac{\calpha \cbeta \gamma_k}{\|X^* \|},
\]
so that
\[
\|X^*\| \geq 4 \calpha \cbeta \gamma_k.
\]
By the definition of $\cR''_3$, we have
\[
\|W^k\| > \frac{20}{19} \sqrt{2} \|X^*\|^{1/2} \geq \frac{40}{19}\sqrt{2} \calpha^{1/2} \cbeta^{1/2} \gamma_k^{1/2}.
\]
Now the strict saddle property, and in particular \eqref{eq:glargegrad2}, together with \eqref{eq:cfac}, implies that
\[
\|\nabla G(W^k)\|_F \geq \cepsilon \|W^k\|^3 \geq \cepsilon \left(\frac{40}{19}\sqrt{2}\right)^3 \calpha^{3/2} \cbeta^{3/2} \gamma_k^{3/2} \geq \cfac \gamma_k^{3/2},
\]
so a gradient step is taken in this case as well.

Finally, assume that $W^k \in \cR'''_3$. By the same logic as the above, we have $\|X^*\| \geq 4 \calpha \cbeta \gamma_k$. From the definition of $\cR'''_3$, we have
\[
\|W^k(W^k)^\top\|_F > \frac{20}{9} \|X^*\|_F  \geq \frac{20}{9} \|X^*\| \geq \frac{80}{9} \calpha \cbeta \gamma_k.
\]
Together with  \eqref{eq:glargegrad3} and \eqref{eq:cfac}, we have
\begin{equation*}
\begin{split}
\|\nabla G(W^k)\|_F \geq \cepsilon\|W^k(W^k)^\top\|_F^{3/2}
&\geq \cepsilon \left(\frac{80}{9}\right)^{3/2} \calpha^{3/2} \cbeta^{3/2} \gamma_k^{3/2} \\
&\geq \cfac \gamma_k^{3/2},
\end{split}
\end{equation*}
so a gradient step is taken in this case too.
\end{proof}

Together, Lemmas~\ref{lem:r2step} and \ref{lem:r3step} imply that once $\gamma_k \in \Gamma(X^*)$, then (with high probabilty) the local phase  (Algorithm~\ref{alg:local}) will be invoked only when $W^k \in \cR_1$. 
With this observation in mind, we focus on the behavior of the local phase when $\gamma_k \in \Gamma(X^*)$.  
We show that when $W^k \in \cR_1$ and $\gamma_k \in \Gamma(X^*)$, then $W^{k+1}$ satisfies the approximate optimality conditions \eqref{eq:optconds}.

\begin{lemma} \label{lem:linearconvergence}
Let Assumptions~\ref{assum:xstar}, \ref{assum:frestrictedsc}, \ref{assum:compactlevelset}, \ref{assum:GC22}, and \ref{assum:lipknown} hold. 
Suppose that $\gamma_k \in \Gamma(X^*)$ hold and that $W^k_0 \in \cR_1$. 
Then, if Algorithm~\ref{alg:local} is invoked by Algorithm~\ref{alg:stupid}, we have for all $t \ge 0$ that $W^k_t$ in Algorithm~\ref{alg:local} satisfies
\begin{subequations}
  \label{eq:localineqshold}
  \begin{align}
    \label{eq:localineqshold.1}
    \|\nabla G(W^k_t)\|_F & \leq \frac{\sqrt{\kappa_t}}{\beta_k} \delta_k, \\
    \label{eq:localineqshold.2}
2 \|\nabla f(X^k_t)\|_F + \muhalf \|(\hat{W}^k_t)^\top W^k_t\|_F & \leq \tau_t.
\end{align}
\end{subequations}
In addition, Algorithm~\ref{alg:local} terminates with the flag ``converged'' set to ``True'' and $W^{k+1}$ satisfying
\[
\|\nabla G(W^{k+1})\|_F \leq \epsg, \quad \lambdamin(\nabla^2 G(W^{k+1})) \geq -\epsH.
\]
\end{lemma}
\begin{proof}
Since $W^k_0 \in \cR_1$ and $\gamma_k \in \Gamma(X^*)$, it follows that
\[
\delta_k := \sqrt{2} \gamma_k^{1/2} \geq \sigma_r^{1/2}(X^*)
\]
and therefore $\dist(W^k_0, W^*) \leq \delta = \sigma_r^{1/2}(X^*) \leq \sqrt{2} \gamma_k^{1/2} = \delta_k$, where $\delta$ is from the $(\alpha, \beta, \delta)$ regularity condition \eqref{eq:gregular} and is defined to be $\sigma_r^{1/2}(X^*)$ in Theorem~\ref{thm:strictsaddle}.  
Let $R = R(W^k_0,W^*)$ be the orthogonal matrix that minimizes $\|W^*R - W^k_0\|_F.$ 
Then, using \eqref{eq:XWstar}, we have
\begin{align*}
\sqrt{2} \|X^*\|^{1/2} = \|W^*\| \leq \|W^*\|_F &= \|W^* R\|_F \\ 
&\leq  \|W^* R - W^k_0\|_F + \|W^k_0\|_F \\ &= \dist(W^k_0,W^*) + \|W^k_0\|_F \\ &\leq \delta_k + \|W^k_0\|_F.
\end{align*}
It follows that
\begin{equation} \label{eq:za4}
\beta_k = \frac{2 \cbeta}{(\delta_k + \|W^k_0\|_F)^2} \leq \frac{\cbeta}{\|X^*\|} = \beta.
\end{equation}
where $\beta$ is from the $(\alpha, \beta, \delta)$ regularity condition \eqref{eq:gregular} and defined to be $\cbeta \|X^*\|^{-1}$ in Theorem~\ref{thm:strictsaddle}. 
Therefore, by taking a stepsize $\nu_t \leq 2\beta_k \leq 2\beta$, it follows from Lemma~\ref{lem:localconvergence} that $W^k_t \in B(W^*, \delta)$ for all $t \geq 0$ and
\begin{equation} \label{eq:lemconvergenceeq1}
\dist^2(W^k_t, W^*) \leq \dist^2(W^k_0, W^*) \prod_{j=0}^{t-1} (1 - 2 \nu_j \alpha)  
\leq  \delta_k^2 \prod_{j=0}^{t-1} (1 - 2 \nu_j \alpha_k) = \kappa_t \delta_k^2,
\end{equation}
where we used $\alpha = \calpha \sigma_r(X^*) \geq \calpha \gamma_k = \alpha_k$, $\dist(W^k_0, W^*) \leq \delta_k$, and the definition of $\kappa_t$. 
Recalling \eqref{eq:gradientUB} and the definition of $\beta$ in our strict saddle conditions, it follows that
\[
\frac{\cbeta}{\|X^*\|} \|\nabla G(W^k_t)\| \leq \dist(W^k_t, W^*), \quad
\mbox{for all $t \ge 0$.}
\]
Together with \eqref{eq:lemconvergenceeq1} this implies
\[
\|\nabla G(W^k_t)\| \leq \sqrt{\kappa_t} \frac{\|X^*\|}{\cbeta} \delta_k \leq
\frac{\sqrt{\kappa_t}}{\beta_k} \delta_k,
\]
where we used $\|X^*\|/\cbeta=1/\beta \le 1/\beta_k$ for the latter inequality, thus proving \eqref{eq:localineqshold.1}. 
To prove \eqref{eq:localineqshold.2} (which holds for $W^k_0$ by our local phase initialization step in Algorithm \ref{alg:stupid}, by the definition of $\tau_0$), we have from Lemma~\ref{lem:ncbound} and \eqref{eq:lemconvergenceeq1} that
\begin{align*}
&2\|\nabla f(X^k_t)\|_F + \muhalf \|(\hat{W^k_t})^\top W^k_t\|_F \\
&\leq (2 L_{\nabla f} + \muhalf) (2\|W^k_t\|_F + \dist(W^k_t,W^*))\dist(W^k_t,W^*) \\
&\leq (2 L_{\nabla f} + \muhalf) (2\|W^k_t\|_F + \sqrt{\kappa_t}\delta_k)\sqrt{\kappa_t}\delta_k
= \tau_t.
\end{align*}

It follows from \eqref{eq:localineqshold} that the ``while'' loop in Algorithm~\ref{alg:local} terminates only when $\|\nabla G(W^k_t)\| \leq \epsg$ and $2\|\nabla f(X^k_t)\|_F + \muhalf \|(\hat{W^k_t})^\top W^k_t\|_F \leq \epsH$.  
Since $0 \le 1 - 2 \nu_{t-1} \alpha_k < 1$, the bounds $0 \leq \kappa_t \leq \kappa_{t-1}$ hold for all $t=1,2,\dotsc$, with equality holding only when $\kappa_{t-1} = 0$. 
Therefore, by \eqref{eq:localineqshold}, it follows by the definition of $\tau_t$ and $\|W^k_t\| \leq R_{\mathcal{L}}$ for all $t \ge 0$ (see \eqref{eq:boundscompact}) that
\[
\lim_{t \rightarrow \infty} \|\nabla G(W^k_t)\| = 0, \quad \lim_{t
  \rightarrow \infty} 2\|\nabla f(X^k_t)\|_F + \muhalf
\|(\hat{W^k_t})^\top W^k_t\|_F = 0,
\]
and thus the while loop must terminate with
\[
\|\nabla G(W^k_{\bar{t}})\| \leq \epsg, \quad
2\|\nabla f(X^k_{\bar{t}})\|_F + \muhalf \|(\hat{W}^k_{\bar{t}})^\top W^k_{\bar{t}}\|_F \leq \epsH,
\]
for some $\bar{t} \geq 0$. Then, by Lemma~\ref{lem:nclowerbound} and
$W^{k+1} = W^k_{\bar{t}}$, the final claim is proved.
\end{proof}


\section{Complexity Analysis.} \label{sec:complexity}

This section presents our complexity results for Algorithm~\ref{alg:stupid}. We provide a brief "roadmap" to the sequence of results here.

We start by showing (Lemma~\ref{lem:parambounds}) how the parameters $\alpha_k$, $\delta_k$, and $\beta_k$ in the algorithm relate to the properties of the objective function and solution, in particular the key quantity $\sigma_r(X^*)$. We follow up with a result (Lemma~\ref{lem:btdecrease}) that shows that the reduction in $G$ from a backtracking line search along the negative gradient direction is a multiple of $\|\nabla G (W) \|^2$, then apply this result to the line searches \eqref{eq:lsdecreasegrad} and \eqref{eq:lsdecreasegradlocal} (see Lemmas~\ref{lem:graddecrease} and \ref{lem:localgraddecrease}, respectively). For backtracking steps along negative curvature directions \eqref{eq:lsdecreasenc}, we show that the reduction in $G$ is a multiple of $\gamma_k^3$ (Lemma~\ref{lem:ncdecrease}).

The next result, Lemma~\ref{lem:maxinneriters},  is a bound on the number of iterations taken in Algorithm~\ref{alg:local} when it is invoked with $\gamma_k \ge \tfrac12 \sigma_r(X^*)$. We then return to the main algorithm, Algorithm~\ref{alg:stupid}, and derive a bound on the number of non-local iterations (negative gradient or negative curvature steps) under the assumptions that $\gamma_k \ge \tfrac12 \sigma_r(X^*)$ and $G$ is bounded below (Lemma~\ref{lem:maxgradnciters}). Lemma~\ref{lem:localtermination} then derives conditions under which a call to Algorithm~\ref{alg:local} will be made that results in successful termination.

Lemma~\ref{lem:mingamma} is a particularly important result, showing that with high probability, we have that $\gamma_k \ge \tfrac12 \sigma_r(X^*)$ at all iterations, and placing a bound on the number of times that Algorithm~\ref{alg:local} is invoked. This result leads into the main convergence results, Theorem~\ref{thm:wcc} and Corollary~\ref{cor:wcc}, which show iteration and complexity bounds for the algorithm.



\subsection{Strict Saddle Parameters.} 

In this subsection, we present lemmas which provide bounds on the parameters $\alpha_k$, $\beta_k$, $\delta_k$, and $\gamma_k$ which are generated throughout Algorithm~\ref{alg:stupid} and are used to estimate the true strict saddle parameters.

\begin{lemma} \label{lem:parambounds}
Let Assumptions \ref{assum:xstar}, \ref{assum:frestrictedsc}, and \ref{assum:compactlevelset} hold.  
Let $\alpha_k$, $\gamma_k$, $\delta_k$, and $\beta_k$ be the values of defined in Algorithm~\ref{alg:stupid} (for those iterations $k$ on which they are defined). 
Then, for any $k$ such that $\gamma_k \geq \frac12 \sigma_r(X^*)$ holds, we have
\begin{subequations} \label{eq:kva}
  \begin{align}
    \label{eq:kva.1}
\frac12 \sigma_r(X^*) &\leq\gamma_k \leq \gamma_0, \\
   \label{eq:kva.2}
\frac{\calpha}{2} \sigma_r(X^*) &\leq \alpha_k \leq \calpha \gamma_0, \\
   \label{eq:kva.3}
\sigma_r^{1/2}(X^*) &\leq \delta_k \leq \sqrt{2} \gamma_0^{1/2}, \\
   \label{eq:kva.4}
\frac{2\cbeta}{\left(\sqrt{2} \gamma_0^{1/2} + R_{\mathcal{L}}\right)^2} &\leq \beta_k \leq \frac{2\cbeta}{\sigma_r(X^*)},
\end{align}
\end{subequations}
where $ R_{\mathcal{L}}$ is defined in \eqref{eq:boundscompact}.
\end{lemma}
\begin{proof}
By the definition of our algorithm, it follows that $\gamma_0 \geq \gamma_k$ for all $k$. In addition, by our assumption that $\gamma_k \geq \frac12 \sigma_r(X^*)$ holds, we have proved \eqref{eq:kva.1}.

Noting that $\alpha_k = \calpha \gamma_k$, \eqref{eq:kva.2} follows directly from \eqref{eq:kva.1}.

Now, for $\delta_k$, we have from $\gamma_k \geq \frac12 \sigma_r(X^*)$ that
\[
\delta_k = \sqrt{2} \gamma_k^{1/2} \geq \sqrt{2} \left(\frac12
\sigma_r(X^*)\right)^{1/2} = \sigma_r^{1/2}(X^*),
\]
while
\[
\delta_k = \sqrt{2} \gamma_k^{1/2}  \leq \sqrt{2} \gamma_0^{1/2},
\]
proving \eqref{eq:kva.3}.

Recalling the definition of $\beta_k$, we have
\[
\beta_k = \frac{2\cbeta}{(\delta_k + \|W^k\|_F)^2} = \frac{2\cbeta}{(\sqrt{2} \gamma_k^{1/2} + \|W^k\|_F)^2}
\leq \frac{2\cbeta}{\sigma_r(X^*)}.
\]
For a lower bound on $\beta_k$, note that the backtracking linesearches at each step ensure monotonicity of the iterates, so that $W^k \in \mathcal{L}_{W^0}$ for all $k \geq 0$. Thus, we have
\[
\beta_k = \frac{2\cbeta}{\left(\delta_k + \|W^k\|_F\right)^2} = \frac{2\cbeta}{\left(\sqrt{2} \gamma_k^{1/2} + \|W^k\|_F\right)^2}
\geq \frac{2\cbeta}{\left(\sqrt{2} \gamma_0^{1/2} + R_{\mathcal{L}}\right)^2},
\]
completing our proof of  \eqref{eq:kva.4}.
\end{proof}

\subsection{Line Search Guarantees.}
We now provide guarantees of termination and descent for the two line searches in Algorithm~\ref{alg:stupid} and the line search in Algorithm~\ref{alg:local}.  
We begin by providing a generic lemma for Armijo backtracking on gradient descent steps.

\begin{lemma} \label{lem:btdecrease}
Suppose that Assumptions~\ref{assum:compactlevelset} and \ref{assum:GC22} hold. 
Suppose that a step is computed from $W$ using a backtracking linesearch along the negative gradient direction 
with a step length $\nu=\zeta \theta^l$, where $l$ is the smallest nonnegative integer such that
\begin{equation} \label{eq:lsdecreasebt}
G(W - \nu \nabla G(W)) < G(W) - \eta \nu \|\nabla G(W)\|_F^2,
\end{equation}
where $\eta \in (0,1)$ is the sufficient decrease parameter. Then, the backtracking line search requires at most $j \leq \jbt+ 1$ iterations, where
\begin{equation} \label{eq:defjbt}
\jbt= \left[\log_\theta\left(\frac{2(1 - \eta)}{L_g \zeta}\right)\right]_+,
\end{equation}
and the resulting step satisfies
\begin{equation} \label{eq:btdecrease}
G(W) - G(W - \nu \nabla G(W)) \geq \cbt \|\nabla G(W)\|^2_F
\end{equation}
where $\nu$ is the step size found by the backtracking procedure and
\begin{equation} \label{eq:defcbt}
\cbt = \eta \min\left\{\zeta, \frac{2 \theta(1 - \eta)}{L_g}\right\}. 
\end{equation}
\end{lemma}
\begin{proof}
Suppose that the maximum steplength is accepted (that is, $\nu = \zeta$). Then,
\[
G(W - \zeta \nabla G(W)) < G(W) - \zeta \eta \|\nabla G(W)\|_F^2
\]
so the claim holds in this case.
For the remainder of the proof, we assume that $\nu < 1$. For any $l \geq 0$ such that \eqref{eq:lsdecreasebt} does not hold, we have from \eqref{eq:Lg} that
\begin{align*}
-\eta \zeta \theta^l \|\nabla G(W)\|_F^2 &\leq
G(W - \zeta \theta^l \nabla G(W)) - G(W) \\
&\leq -\zeta \theta^l \langle \nabla G(W), \nabla G(W) \rangle
+ \frac{L_g \zeta^2 \theta^{2l}}{2} \|\nabla G(W)\|_F^2 \\
&= -\zeta \theta^l \left(1 - \frac{L_g \zeta \theta^l}{2}\right) \|\nabla G(W)\|_F^2.
\end{align*}
By rearranging this expression, we obtain
\[
\frac{L_g}{2} \zeta \theta^l \geq 1 - \eta.
\]
Therefore, for any $l \geq 0$ where \eqref{eq:lsdecreasebt} is not satisfied, we have
\begin{equation} \label{eq:thetalbbtlem}
\theta^l \geq \frac{2(1 - \eta)}{L_g \zeta}
\end{equation}
holds. For any $l > \jbt$ we have
\[
\theta^l < \theta^{\jbt} \leq \frac{2(1 - \eta)}{L_g \zeta}
\]
so \eqref{eq:thetalbbtlem} cannot be satisfied for any $l > \jbt$ and the line search must terminate with $\nu = \zeta \theta^j$ for some $1 \leq j \leq \jbt + 1$. 
The value of this index prior to backtracking termination satisfies \eqref{eq:thetalbbtlem}, so we have
\[
\theta^j \geq \frac{2 \theta (1 - \eta)}{L_g \zeta}.
\]
Thus,
\[
G(W) - G(W - \nu \nabla G(W)) \geq  \eta \zeta \theta^j \|\nabla G(W)\|^2
\geq \frac{2 \eta \theta(1 - \eta)}{L_g} \|\nabla G(W)\|^2.
\]
\end{proof}

Lemma~\ref{lem:btdecrease} is used directly in  the next two results, which provide termination and decrease guarantees for the linesearches used in the large gradient case and in the local phase.

\begin{lemma} \label{lem:graddecrease}
Suppose that Assumptions~\ref{assum:compactlevelset} and \ref{assum:GC22} hold. 
Suppose that the backtracking step \eqref{eq:lsdecreasegrad} is taken at outer iteration $k$. 
Then, the backtracking line search requires at most $j_k \leq \jgrad + 1$ iterations, where
\begin{equation} \label{eq:defjgrad}
\jgrad := \left[\log_\theta\left(\frac{2(1 - \eta)}{L_g}\right)\right]_+,
\end{equation}
and the resulting step satisfies $W^{k+1} = W^k - \nu_k \nabla G(W^k)$
\begin{equation} \label{eq:graddecrease}
G(W^k) - G(W^{k+1}) \geq \cgrad \|\nabla G(W^k)\|^2_F
\end{equation}
where
\begin{equation} \label{eq:defcgrad}
\cgrad = \eta \min\left\{1, \frac{2 \theta(1 - \eta)}{L_g}\right\}. 
\end{equation}
\end{lemma}
\begin{proof}
This proof follows directly from Lemma~\ref{lem:btdecrease} with
$\zeta = 1$.
\end{proof}

\begin{lemma} \label{lem:localgraddecrease}
Suppose that Assumptions~\ref{assum:xstar}, \ref{assum:frestrictedsc}, \ref{assum:compactlevelset}, and \ref{assum:GC22} hold. 
Suppose that the backtracking step \eqref{eq:lsdecreasegradlocal} is taken at inner iteration $t$ of outer iteration $k$, and that $\gamma_k \ge \tfrac12 \sigma_r (X^*)$.
Then, the backtracking line search requires at most $j_t \leq \hat{j}_k + 1$ iterations, where
\begin{equation} \label{eq:defjlocal}
\hat{j}_k := \left[\log_\theta\left(\frac{1 - \eta}{L_g \beta_k}\right)\right]_+,
\end{equation}
the resulting step satisfies $W^k_{t+1} = W^k_t - \nu_t \nabla G(W^k_t)$, and 
\begin{equation} \label{eq:localgraddecrease}
G(W^k_t) - G(W^k_{t+1}) \geq \clocal \|\nabla G(W^k_t)\|^2_F
\end{equation}
where
\[
\clocal = \eta \min\left\{\frac{4\cbeta}{(\sqrt{2}\gamma_0^{1/2} +
  R_{\mathcal{L}})^2}, \frac{2\theta(1 - \eta)}{L_g}\right\}.
\]
\end{lemma}
\begin{proof}
This result follows from Lemma~\ref{lem:btdecrease} with $\zeta = 2\beta_k$ and
\[
\beta_k \geq \frac{2\cbeta}{(\sqrt{2}\gamma_0^{1/2} + R_{\mathcal{L}})^2},
\]
which follows from Lemma~\ref{lem:parambounds}.
\end{proof}

Next, we provide similar guarantees for the negative curvature linesearch in Algorithm~\ref{alg:stupid}.

\begin{lemma} \label{lem:ncdecrease}
Suppose that Assumptions~\ref{assum:compactlevelset} and \ref{assum:GC22}
hold. Suppose that the backtracking step \eqref{eq:lsdecreasenc} is
taken at outer iteration $k$. Then, the backtracking
line search requires at most $j_k \leq \jnc + 1$ iterations, where
\begin{equation} \label{eq:defjnc}
\jnc:= \left[\log_\theta\left(\frac{3(1 - \eta)}{L_H}\right)\right]_+,
\end{equation}
and the resulting step satisfies $W^{k+1} = W^k + \nu_k D^k$
\begin{equation} \label{eq:ncdecrease}
G(W^k) - G(W^{k+1}) \geq \cnc \gamma_k^3
\end{equation}
where
\begin{equation} \label{eq:defcnc}
\cnc = \frac{\eta \cgamma^3}{16} \min\left\{1, \left(\frac{3 \theta (1 - \eta)}{L_H} \right)^2\right\}. 
\end{equation}
\end{lemma}
\begin{proof}
First, by the scaling applied to $D^k$ in Algorithm~\ref{alg:stupid}, it follows that
\begin{equation} \label{eq:dkcurvaturecond}
-\|D^k\|_F^3 = \langle D^k, \nabla^2 G(W^k) D^k \rangle \leq -\frac18 \cgamma^3 \gamma_k^3,
\end{equation}
where the last inequality follows from $\langle S, \nabla^2 G(W^k) S \rangle \leq -\frac12 \cgamma \gamma_k$ and $\|S\|_F = 1$. In addition, we have
\begin{equation} \label{eq:dkgradcond}
\langle \nabla G(W^k), D^k \rangle \leq 0.
\end{equation}

Suppose that the unit step is accepted (that is, $\nu_k = 1$). Then
\[
G(W^k + D^k) -G(W^k) < \frac{\eta}{2} \langle D^k, \nabla^2 G(W^k) D^k \rangle
\leq - \frac{\eta}{16} \cgamma^3 \gamma_k^3 \le -\cnc \gamma_k^3,
\]
holds so the claim holds in this case.  For the remainder of the proof, we assume that $\nu_k < 1$, that is, $j_k \ge 1$. For any $j \geq 0$ such that \eqref{eq:lsdecreasenc} does not hold, we have from \eqref{eq:LH}, \eqref{eq:dkcurvaturecond}, and \eqref{eq:dkgradcond} that 
\begin{align*}
-\frac{\eta \theta^{2j}}{2} \|D^k\|_F^3 
& =\frac{\eta \theta^{2j}}{2} \langle D^k, \nabla^2 G(W^k) D^k \rangle \\
&\leq G(W^k + \theta^j D^k) - G(W^k) \\
&\leq \theta^j \langle \nabla G(W^k), D^k \rangle
+ \frac{\theta^{2j}}{2} \langle D^k, \nabla^2 G(W^k) D^k \rangle + \frac{L_H \theta^{3j}}{6} \|D^k\|_F^3 \\
&\le -\frac{\theta^{2j}}{2} \|D^k\|_F^3
+ \frac{L_H \theta^{3j}}{6} \|D^k\|_F^3 \\
&= -\frac{\theta^{2j}}{2}\left(1 - \frac{L_H \theta^{j}}{3}\right) \|D^k\|_F^3 \\
\end{align*}
By rearranging this expression, we have
\begin{equation} \label{eq:thetalbnclem}
  \frac{L_H}{3} \theta^j \geq 1 - \eta \;\; \Rightarrow \;\;
  \theta^j \geq \frac{3(1 - \eta)}{L_H}.
\end{equation}
For any $j > \jnc$ we have
\[
\theta^j < \theta^{\jnc} \leq \frac{3(1 - \eta)}{L_H}
\]
so \eqref{eq:thetalbnclem} cannot be satisfied for any $j > \jnc$ and the line search must terminate with $\nu_k = \theta^{j_k}$ for some $1 \leq j_k \leq \jnc + 1$. The value $j=j_k-1$ satisfies \eqref{eq:thetalbnclem}, so we have
\[
\theta^{j_k} \geq \frac{3 \theta (1 - \eta)}{L_H}.
\]
Thus, by \eqref{eq:dkcurvaturecond}, we have
\[
G(W^k) - G(W^{k+1}) \geq -\frac{\eta}{2} \theta^{2j_k} \langle D^k,
\nabla^2 G(W^k) D^k \rangle \geq \frac{\eta}{16} \left(\frac{3 \theta
  (1 - \eta)}{L_H}\right)^2 \cgamma^3 \gamma_k^3 \ge \cnc \gamma_k^3.
\]
Thus, the claim holds in the case of $\nu_k<1$ also, completing the proof.
\end{proof}

\subsection{Properties of Algorithm~\ref{alg:local}.}

This section provides a bound on the maximum number of inner iterations that may occur during the local phase,  Algorithm~\ref{alg:local}.


\begin{lemma} \label{lem:maxinneriters}
Let Assumptions~\ref{assum:xstar}, \ref{assum:frestrictedsc}, \ref{assum:compactlevelset}, \ref{assum:GC22}, and \ref{assum:lipknown} hold and define
\begin{equation} \label{eq:numindef}
\numin := \frac{2 \theta(1-\eta)}{L_g}.
\end{equation}
Then, for all $k$ such that $\gamma_k \geq \frac12 \sigma_r(X^*)$ holds, the ``while'' loop in Algorithm~\ref{alg:local} terminates in at most
\begin{equation} \label{eq:Tdef}
\T := 2 \frac{\log\hat{C} + \log(\max(\epsg^{-1},\epsH^{-1}))}{\log(1/(1-\numin \calpha \sigma_r(X^*)))}
\end{equation}
iterations, where
\begin{equation} \label{eq:chatdef}
\hat{C} := \max\left\{\frac{\gamma_0^{1/2} \left(\sqrt{2} \gamma_0^{1/2} + R_\mathcal{L} \right)^2}{2 \cbeta},
(2 L_{\nabla f} + 1/2) \left(2R_{\mathcal{L}} + \sqrt{2} \gamma_0^{1/2} \right) \sqrt{2} \gamma_0^{1/2}\right\}.
\end{equation}
\end{lemma}
\begin{proof}
By the result of Lemma~\ref{lem:localgraddecrease}, the backtracking line search terminates in at most $\hat{j}_k + 1$ iterations, where $\hat{j}_k$ is defined in \eqref{eq:defjlocal}. 
From this definition, we have
\begin{equation} \label{eq:bv3}
\nu_t \geq 2 \beta_k \theta^{\hat{j}_k+1} \geq \frac{2
  \theta(1-\eta)}{L_g} = \numin, \quad \mbox{for all $t \ge 0$.}
\end{equation}

Assume for contradiction that Algorithm~\ref{alg:local} does not terminate on or before iteration $\T$. Then,
\[
\|\nabla G(W^k_{\T})\|_F
> \epsg \;\; \mbox{and/or} \;\; 2\|\nabla f(X^k_{\T})\|_F + \muhalf
\|(\hat{W}^k_{\T})^\top W^k_{\T}\|_F > \epsH
\]
hold for $t=\T$, and for the tests at the start of the ``while'' loop of Algorithm~\ref{alg:local}, we have that
\begin{subequations} \label{eq:vr1}
\begin{align}
\label{eq:vr1.a}
    \|\nabla G(W^k_t)\|_F & \leq
\frac{\sqrt{\kappa_t}}{\beta_k} \delta_k, \quad \mbox{for all $t=0,1,\dotsc,\T$,} \\
\label{eq:vr1.b}
2\|\nabla f(X^k_t)\|_F + \muhalf \|(\hat{W}^k_t)^\top W^k_t\|_F & \leq \tau_t, \quad \mbox{for all $t=0,1,\dotsc,\T$.}
\end{align}
\end{subequations}
From \eqref{eq:bv3} and Lemma~\ref{lem:parambounds}, we have $1-2\nu_t \alpha_k \leq 1-\numin \calpha \sigma_r(X^*)$ for all $t$. 
From this observation together with $\nu_t \le 2 \beta_k$ and $\alpha_k \beta_k \le 1/4$, we have
\[
0 < \numin \calpha \sigma_r(X^*) \leq 2 \nu_t \alpha_k \leq 4 \alpha_k \beta_k
\leq 1,
\]
so that $1 - \numin \calpha \sigma_r(X^*) \in [0, 1)$.  Thus,
\begin{equation} \label{eq:ut4}
    \kappa_{\T} = \prod_{t=0}^{\T-1} (1 - 2\nu_t \alpha_k) \leq
\prod_{t=0}^{\T-1} (1-\numin \calpha \sigma_r(X^*)) = (1-\numin
\calpha \sigma_r(X^*))^{\T}.
\end{equation}

Consider first the case in which termination does not occur at the ``if'' statement in iteration $\T$ because $\| \nabla G(W^k_\T) \|_F > \epsg$.  
We then have

\begin{align*}
\epsg &< \|\nabla G(W^k_{\T})\|_F \\
&\leq \frac{\sqrt{\kappa_{\T}}}{\beta_k} \delta_k \\
&\leq (1-\numin \calpha \sigma_r(X^*))^{\T/2} \frac{\delta_k}{\beta_k} \\
&\leq (1-\numin \calpha \sigma_r(X^*))^{\T/2}
\frac{\gamma_0^{1/2}}{2\cbeta} (\sqrt{2} \gamma_0^{1/2} + R_\mathcal{L})^2,
\end{align*}
where the final inequality follows from Lemma~\ref{lem:parambounds}. Noting that $1/(1-\numin \calpha \sigma_r(X^*)) \geq 1$, we have by manipulation of this inequality that
\[
\T < 2 \log\left(\frac{\gamma_0^{1/2} (\sqrt{2} \gamma_0^{1/2} + R_\mathcal{L})^2}{2\cbeta \epsg}\right)/
\log(1/(1-\numin \calpha \sigma_r(X^*)))
\]
which implies that 
\[
\T < 2\frac{\log\hat{C} + \log (\epsg^{-1})}{\log(1/(1-\numin\calpha\sigma_r(X^*)))},
\]
which contradicts the definition of $\T$.

The second possibility is that Algorithm~\ref{alg:local} fails to terminate in the ``if'' statement in iteration $\T$ because $2 \| \nabla f(X^k_\T) \|_F + \muhalf \| (\hat{W}^k_\T)^T W^k_\T \|_F > \epsH$.
In this case, we have
\begin{align*}
\epsH &< 2\|\nabla f(X^k_\T)\|_F + \frac12
\|(\hat{W}^k_\T)^\top W^k_\T\|_F  \\
&\leq \tau_\T = (2 L_{\nabla f} + 1/2)(2\|W^k_\T\|_F +
\sqrt{\kappa_\T} \delta_k) \sqrt{\kappa_\T} \delta_k  && \quad \mbox{from \eqref{eq:vr1.b}} \\
&\leq (2 L_{\nabla f} + 1/2)(2\|W^k_\T\|_F + \delta_k)
\sqrt{\kappa_\T} \delta_k && \quad \mbox{since $\kappa_\T \le 1$} \\
&\leq (2 L_{\nabla f} + 1/2)(2\|W^k_\T\|_F + \delta_k)
(1-\numin \calpha \sigma_r(X^*))^{\T/2} \delta_k && \quad
\mbox{from \eqref{eq:ut4}} \\
&\leq (1-\numin \calpha \sigma_r(X^*))^{\T/2} (2 L_{\nabla f} + 1/2)
(2R_{\mathcal{L}} + \sqrt{2} \gamma_0^{1/2}) \sqrt{2} \gamma_0^{1/2},
\end{align*}
where the final inequality follows from $\delta_k = \sqrt{2} \gamma_k^{1/2} \le \sqrt{2} \gamma_0^{1/2}$, Assumptions~\ref{assum:compactlevelset} and \ref{assum:GC22}, and \eqref{eq:boundscompact}. 
By manipulating this inequality and recalling that $1/(1-\numin \calpha \sigma_r(X^*)) \geq 1$, we find that this bound implies
\[
\T < 2 \log\left(\frac{(2 L_{\nabla f} + 1/2) (2R_{\mathcal{L}}
+ \sqrt{2} \gamma_0^{1/2}) \sqrt{2} \gamma_0^{1/2}}{\epsH}\right)/
\log(1/(1-\numin \calpha \sigma_r(X^*)))
\]
so that (similarly to the above)
\[
\T < 2\frac{\log\hat{C} + \log(\epsH^{-1})}{\log(1/(1-\numin \calpha \sigma_r(X^*)))},
\]
which again contradicts the definition of $\T$.

Since both cases lead to a contradiction, our assumption that Algorithm~\ref{alg:local} does not terminate on or before iteration $\T$ cannot be true, and the result is proved.
\end{proof}


\subsection{Worst Case Complexity of Algorithm~\ref{alg:stupid}.}
\label{subsec:wcc}
We now work toward our main complexity result, Theorem~\ref{thm:wcc}. 
We begin with a lemma which bounds the maximum number of large gradient and/or negative curvature iterations that can occur while $\gamma_k \geq \frac12 \sigma_r(X^*)$.

\begin{lemma} \label{lem:maxgradnciters}
Suppose that Assumptions \ref{assum:xstar}, \ref{assum:frestrictedsc}, \ref{assum:compactlevelset}, and \ref{assum:GC22} hold. 
Let Algorithm~\ref{alg:stupid} be invoked with $\gamma_0 \geq \sigma_r(X^*)$.
Then, while $\gamma_k \geq \frac12 \sigma_r(X^*)$, Algorithm~\ref{alg:stupid} takes at most
\begin{equation} \label{eq:klargedef}
\Klarge := \frac{8(G(W^0) - \Glow)}{\min\left\{\cfac^2 \cgrad, \cnc\right\} \sigma_r(X^*)^3}
\end{equation}
large gradient steps and/or large negative curvature steps.
\end{lemma}
\begin{proof}
We partition the iteration indices $k$ that are used in Algorithm~\ref{alg:stupid} prior to termination as follows: $K_1$ contains those iteration indices for which a large gradient step is taken, $K_2$ contains those for which a large negative curvature step is taken, and $K_3$ contains those for which the local phase is initialized.

By Lemma~\ref{lem:graddecrease} we have for all $k \in K_1$ that 
\[
G(W^k) - G(W^{k+1}) \geq \cgrad \|\nabla G(W^k)\|_F^2 \geq \cfac^2 \cgrad \gamma_k^3,
\]
where $\cgrad$ is defined in \eqref{eq:defcgrad}. 
Similarly, by Lemma~\ref{lem:ncdecrease}, for all $k \in K_2$, we have 
\[
G(W^k) - G(W^{k+1}) \geq \cnc \gamma_k^3,
\]
where $\cnc$ is defined in \eqref{eq:defcnc}.

Now, consider $k \in K_3$. On iterations where the local phase is initialized but not invoked (that is, the condition in the ``if'' statement immediately prior to the call to Algorithm~\ref{alg:local} is not satisfied), then $G(W^k) - G(W^{k+1}) = 0$.
On iterations where the local phase is invoked, by the definition of $T_k$ in Algorithm~\ref{alg:stupid} and the result of Lemma~\ref{lem:localgraddecrease}, it follows that
\[
G(W^k) - G(W^{k+1}) = \sum_{t=0}^{T_k-1} G(W^k_t) - G(W^k_{t+1}) \geq \sum_{t=0}^{T_k-1} \clocal \|\nabla G(W_t^k)\|^2_F \geq 0.
\]
Thus, $G(W^k) - G(W^{k+1}) \geq 0$ holds for all $k \in K_3$.

By defining $K = K_1 \cup K_2 \cup K_3$, we have
\begin{align*}
G(W^0) - G(W^{|K|}) &= \sum_{i=0}^{|K|} (G(W^i) - G(W^{i+1})) \\
&\geq \sum_{i \in K_1} (G(W^i) - G(W^{i+1})) + \sum_{j \in K_2} (G(W^j) - G(W^{j+1})) \\
&\geq \sum_{i \in K_1} \cfac^2 \cgrad \gamma_i^3 + \sum_{j \in K_2} \cnc \gamma_j^3 \\
&\geq \sum_{k \in K_1 \cup K_2} \min\{\cfac^2 \cgrad, \cnc\} \gamma_k^3.
\end{align*}
By assumption, we have $\gamma_k \geq \frac12 \sigma_r(X^*)$, so that
\[
G(W^0) - G(W^{|K|}) \geq \frac18 \min\{\cfac^2 \cgrad, \cnc\} \sigma_r(X^*)^3 |K_1 \cup K_2|.
\]
Since $G(W^{|K|}) \ge \Glow$, we have
\[
|K_1 \cup K_2| \le \frac{8(G(W^0) - \Glow)}{\min\{\cfac^2 \cgrad,\cnc\} \sigma_r(X^*)^3} = \Klarge,
\]
proving our claim.
\end{proof}


Next, we show that if $\gamma_k$ is close to $\sigma_r(X^*)$ and $W^k$ is in the region $\cR_1$, then provided that Procedure~\ref{alg:meo} certifies a near-positive-definite Hessian, Algorithm~\ref{alg:local} will be called and successful termination of Algorithm~\ref{alg:stupid} will ensue.

\begin{lemma} \label{lem:localtermination}
Let Assumptions~\ref{assum:xstar}, \ref{assum:frestrictedsc}, \ref{assum:compactlevelset}, \ref{assum:GC22}, and \ref{assum:lipknown} hold. 
At iteration $k$, suppose that both $\gamma_k \in \Gamma(X^*)$ and $W^k \in \cR_1$ hold and that Procedure~\ref{alg:meo} certifies that\\ $\lambdamin(\nabla^2 G(W^k)) \geq -\cgamma \gamma_k$. 
Then Algorithm~\ref{alg:stupid} terminates at a point $W^{k+1}$ that satisfies approximate optimality conditions \eqref{eq:optconds}.
\end{lemma}
\begin{proof}
By the definitions of $\Gamma(X^*)$, $\alpha_k$, and $\delta_k$, it follows that $\alpha_k = \calpha \gamma_k \leq \calpha \sigma_r(X^*) = \alpha$, and $\delta_k = \sqrt{2} \gamma_k^{1/2} \geq \sigma_r^{1/2}(X^*) = \delta$.
Letting $R=R(W^k,W^*) \in \mathcal{O}_r$ be the orthogonal matrix that minimizes $\|W^*R - W^k\|_F$, we have from \eqref{eq:XWstar} and $W^k \in \cR_1$ that
\begin{align*}
\sqrt{2} \|X^*\|^{1/2} = \|W^*\| \leq \|W^*\|_F &= \|W^* R\|_F \\ 
&\leq  \|W^* R - W^k\|_F + \|W^k\|_F \\ 
&= \dist(W^k,W^*) + \|W^k_0\|_F \\ 
&\leq \delta_k + \|W^k\|_F
\end{align*}
so that
\[
\beta_k = \frac{2\cbeta}{(\delta_k + \|W^k_0\|_F)^2} \leq \frac{\cbeta}{\|X^*\|} = \beta.
\]
Since $\alpha \beta \leq \frac14$ holds by definition, it follows that $\alpha_k \beta_k \leq \frac14$ is satisfied so that the first condition of the ``if'' statement prior to the local phase of Algorithm~\ref{alg:stupid} holds. 
Now, by \eqref{eq:gradientUB} and $W^k \in \cR_1$, $\|\nabla G(W^k)\| \leq \dist(W^k, W^*)/\beta$ holds. Thus,
\[
\|\nabla G(W^k)\| \leq \frac{\delta}{\beta} \leq \frac{\delta_k}{\beta_k},
\]
is satisfied, so the second condition of the ``if'' statement prior to the local phase of Algorithm~\ref{alg:stupid} also holds.
Finally, by Lemma \ref{lem:ncbound} and $\dist(W^k, W^*) \leq \delta_k$, we have
\[
 2\|\nabla f(X^k)\|_F + \muhalf \|(\hat{W}^k)^\top W^k\|_F
 \leq \left(2L_{\nabla f} + 1/2 \right)(2\|W^k\|_F + \delta_k)\delta_k,
\]
so that the final condition of the ``if'' statement also holds and Algorithm~\ref{alg:local} will be invoked at $W^k$.
Thus, by Lemma~\ref{lem:linearconvergence}, Algorithm \ref{alg:stupid} terminates at $W^{k+1}$ that satisfies (\ref{eq:optconds}).
\end{proof}

Now we show that with high probability, $\gamma_k \geq \tfrac12 \sigma_r(X^*)$ holds for all $k$.

\begin{lemma} \label{lem:mingamma}
Suppose that Assumptions~\ref{assum:xstar}, \ref{assum:frestrictedsc}, \ref{assum:compactlevelset}, \ref{assum:GC22}, and \ref{assum:lipknown} hold. 
Let Algorithm~\ref{alg:stupid} be invoked with $\gamma_0 \geq \sigma_r(X^*)$.  
Then, with probability at least 
$(1-\rho)^{\Klarge}$ (where $\Klarge$ is defined in Lemma~\ref{lem:maxgradnciters}), we have that
\[
\gamma_k \geq \frac12 \sigma_r(X^*), \quad \mbox{for all $k$,}
\]
and Algorithm~\ref{alg:local} is invoked at most
\begin{equation} \label{eq:def.Klocal}
\Klocal := \log_2\left(\frac{2 \gamma_0}{\sigma_r(X^*)}\right) \quad \mbox{times.}
\end{equation}
\end{lemma}
\begin{proof}
By the definitions of Algorithm~\ref{alg:stupid} and $\Gamma(X^*)$, it is clear that $\gamma_j < \frac12 \sigma_r(X^*)$ can only occur at an iteration $j$ such that $\gamma_k \in \Gamma(X^*)$ holds for some $k < j$.
Let $\hat{K}$ denote the set of (consecutive) iterations for which  $\gamma_k \in \Gamma(X^*)$.

Consider any iteration $k \in \hat{K}$.
Due to the structure of Algorithm \ref{alg:stupid}, $\gamma_k$ will be halved only on iterations $k$ for which $\|\nabla G(W^k)\|_F < \cfac \gamma_k^{3/2}$ is satisfied and Procedure~\ref{alg:meo} certifies that $\lambdamin(\nabla^2 G(W^k)) \geq -\cgamma \gamma_k$.
From Lemma \ref{lem:r3step}, $\|\nabla G(W^k)\| < \cfac \gamma_k^{3/2}$ cannot hold for $W^k \in \cR_3$ and $\gamma_k \in \Gamma(X^*)$, so $\gamma_k$ cannot be halved on such iterations. 
Next, consider $k \in \hat{K}$ such that $W^k \in \cR_1$ and Procedure~\ref{alg:meo} certifies that $\lambdamin(\nabla^2 G(W^k)) \geq -\cgamma \gamma_k$. 
In this case, Algorithm~\ref{alg:stupid} terminates at $W^{k+1}$, by Lemma~\ref{lem:localtermination}. 
Thus, it follows that $\gamma_k$ can be reduced to a level below $\frac12 \sigma_r(X^*)$  only if there is some iteration $k \in \hat{K}$ such that $W^k \in \cR_2$, $\|\nabla G(W^k)\|_F < \cfac \gamma_k^{3/2}$ and Procedure~\ref{alg:meo} certifies that $\lambdamin(\nabla^2 G(W^k)) \geq -\cgamma \gamma_k$. 
Since for $\gamma_k \in \Gamma(X*)$ and $W^k \in \cR_2$, we have $\lambdamin(\nabla^2 G(W^k)) \le -\cgamma \sigma_r(X^*) < -\cgamma \gamma_k$, this ``certification'' by Procedure~\ref{alg:meo} would be erroneous, an event that happens with probability at most $\rho$. Otherwise, a negative curvature backtracking step is taken.

Since the  maximum number of large negative curvature steps that can occur while $\gamma_k \in \Gamma(X^*)$ is bounded by $\Klarge$ (Lemma \ref{lem:maxgradnciters}), there are at most $\Klarge$ iterations for which both $W^k \in \cR_2$ and $\gamma_k \in \Gamma(X^*)$ hold.
It follows that with probability at least $(1-\rho)^{\Klarge}$, Procedure \ref{alg:meo} does not certify that $\lambdamin(\nabla^2 G(W^k) \geq -\cgamma \gamma_k$ while $W^k \in \cR_2$ for all $k \in \hat{K}$. 
This further implies that with probability at least $(1-\rho)^{\Klarge}$, $\gamma_k \geq \frac12 \sigma_r(X^*)$ holds for all $k$.

The second claim follows immediately from the first claim together with the facts that Algorithm~\ref{alg:local} is invoked at least once for each value of $\gamma_k$, and that successive values of $\gamma_k$ differ by factors of 2.
\end{proof}

We are now ready to state our iteration complexity result.

\begin{theorem} \label{thm:wcc}
Suppose that Assumptions~\ref{assum:xstar}, \ref{assum:frestrictedsc}, \ref{assum:compactlevelset}, \ref{assum:GC22}, and \ref{assum:lipknown} hold. 
Then, with probability at least $(1-\rho)^{\Klarge}$, Algorithm~\ref{alg:stupid} terminates in at most
\begin{equation} \label{eq:outeriterations}
\Kouter := \Klarge + \Klocal 
\end{equation}
outer iterations (where $\Klarge$ and $\Klocal$ are defined in \eqref{eq:klargedef} and \eqref{eq:def.Klocal}, resp.) and
\begin{equation} \label{eq:totaliterations}
\Ktotal := \Klarge
+ 2 \frac{\log\hat{C} + \log \max (\epsg^{-1},\epsH^{-1})}{
\log(1/(1-\numin \calpha \sigma_r(X^*)))} \Klocal
\end{equation}
total iterations at a point satisfying \eqref{eq:optconds}, where  $\hat{C}$ is defined in \eqref{eq:chatdef} and $\numin$ is defined in \eqref{eq:numindef}.
\end{theorem}
\begin{proof}
By Lemma~\ref{lem:mingamma}, with probability $(1-\rho)^{\Klarge}$, $\gamma_k \geq \frac12 \sigma_r(X^*)$ holds for all $k$ and Algorithm~\ref{alg:stupid} invokes Algorithm~\ref{alg:local} at most $\Klocal$ times.
Thus, by Lemma~\ref{lem:maxgradnciters}, it follows that Algorithm~\ref{alg:stupid} takes at most
$\Klarge$ large gradient steps and/or large negative curvature iterations with probability at least $(1-\rho)^{\Klarge}$. 
Altogether, this implies that with probability at least $(1-\rho)^{\Klarge}$, the maximum number of outer iterations in Algorithm~\ref{alg:stupid} is bounded by $\Klarge + \Klocal$, proving \eqref{eq:outeriterations}.

The bound \eqref{eq:totaliterations} follows by combining Lemma~\ref{lem:maxinneriters} (which bounds the number of iterations in Algorithm~\ref{alg:local}  at each invocation) with the fact that Algorithm~\ref{alg:local}  is involved at most $\Klocal$ times, with probability at least $(1-\rho)^{\Klarge}$.
\end{proof}

We turn our attention now to providing a complexity in terms of gradient evaluations and/or Hessian vector products. 
One more assumption is needed on Procedure~\ref{alg:meo}.

\begin{assumption} \label{assum:wccmeo}
For every iteration $k$ at which Algorithm~\ref{alg:stupid} calls Procedure~\ref{alg:meo}, and for a specified failure probability $\rho$ with $0 \le \rho \ll 1$, Procedure~\ref{alg:meo} either certifies that $\nabla^2 G(W^k) \succeq -\epsilon I$ or finds a vector of curvature smaller than $-{\epsilon}/{2}$ in at most
\begin{equation} \label{eq:wccmeo}
		\Nmeo:=\min\left\{N, 1+\left\lceil\Cmeo\epsilon^{-1/2}\right\rceil\right\}
\end{equation}
Hessian-vector products (where $N=(m+n)r$ is the number of elements in $W^k$), with probability $1-\rho$, where $\Cmeo$ depends at most logarithmically on $\rho$ and $\epsilon$.
\end{assumption}

Assumption~\ref{assum:wccmeo} encompasses the strategies we mentioned in Section~\ref{subsec:meo}. 
Assuming the bound $U_H$ on $\|\nabla^2 G(W)\|$ to be available, for the Lanczos method with a random starting vector,
\eqref{eq:wccmeo} holds with $\Cmeo=\ln(2.75(nr+mr)/\rho^2)\sqrt{U_H}/2$. 
When a bound on $\|\nabla^2 G(W)\|$ is not available in advance, it can be estimated efficiently with minimal effect on the overall complexity of the method, see Appendix B.3 of~\cite{CWRoyer_MONeill_SJWright_2019}.

Under this assumption, we have the following corollary regarding the maximum number of gradient evaluations/Hessian vector products required by Algorithm~\ref{alg:stupid} to find a point satisfying our approximate second-order conditions \eqref{eq:optconds}.

\begin{corollary} \label{cor:wcc}
Suppose that the assumptions of Theorem~\ref{thm:wcc} are satisfied, and that Assumption~\ref{assum:wccmeo} is also satsified with $\Nmeo$ defined in \eqref{eq:wccmeo}.  
Then, with probability $(1-\rho)^{\Klarge}$, the number of gradient evaluations and/or Hessian-vector products required by Algorithm~\ref{alg:stupid} to output an iterate satisfying \eqref{eq:optconds} is at most
\begin{equation} \label{eq:sj7}
    \Nmeo \left(\Klarge +\Klocal \right)
+ 2 \frac{\log\hat{C} + \log\max(\epsg^{-1},\epsH^{-1})}{\log(1/(1-\numin \calpha \sigma_r(X^*)))} \Klocal,
\end{equation}
and $\Nmeo$ satisfies the upper bound
\[
\Nmeo \leq \min\left\{(n+m)r, 1+\left\lceil \sqrt{2}\Cmeo\cgamma^{-1/2}\sigma_r^{-1/2}(X^*)\right\rceil\right\}.
\]
\end{corollary}
\begin{proof}
All large gradient steps and local iterations (iterations in Algorithm~\ref{alg:local}) require a single gradient evaluation.  
Thus, the number of gradient evaluations in the local phase is equal to the number of iterations in this phase.

Procedure~\ref{alg:meo} is invoked at every large negative curvature iteration and before each time the local phase is tried. With probability $(1-\rho)^{\Klarge}$, the maximum number of large negative curvature iterations is bounded by $\Klarge$ while the maximum number of times the local phase is entered is bounded by $\Klocal$.  
In addition, by Assumption~\ref{assum:wccmeo}, Procedure~\ref{alg:meo} requires at most $\Nmeo$ Hessian-vector products. 
Thus, the maximum number of gradient evaluations and/or Hessian-vector products required is bounded by the quantity in \eqref{eq:sj7}.

Since $\gamma_k \geq \frac12 \sigma_r(X^*)$ with probability $(1-\rho)^{\Klarge}$, it follows that
\begin{equation*}
\begin{split}
\Nmeo &= \min\left\{N, 1+\left\lceil\Cmeo\cgamma^{-1/2}\gamma^{-1/2}_k\right\rceil\right\} \\
&\leq \min\left\{(n+m)r, 1+\left\lceil \sqrt{2}\Cmeo\cgamma^{-1/2}\sigma_r^{-1/2}(X^*)\right\rceil\right\},
\end{split}
\end{equation*}
verifying the bound on $\Nmeo$.
\end{proof}

\section{Conclusion.}
We have described an algorithm that finds an approximate second-order point for robust strict saddle functions.
This method does not require knowledge of the strict saddle parameters that define the optimization landscape or a specialized initialization procedure. By contrast with other methods proposed recently for finding approximate second-order points for nonconvex smooth functions (see, for example, \cite{YCarmon_JCDuchi_OHinder_ASidford_2018,CWRoyer_MONeill_SJWright_2019}), the complexity is not related to a negative power of the optimality tolerance parameter, but depends only logarithmically on this quantity.
The iteration complexity and the gradient complexity depend instead on a negative power of $\sigma_r(X^*)$, the smallest nonzero singular value of the (rank-$r$) minimizer of $f$.


One future research direction lies in investigating whether accelerated gradient methods are suitable for use in the local phase of Algorithm~\ref{alg:stupid}. 
While effective in practice \cite{EJRPauwels_ABeck_YCEldar_SSabach_2017}, little is known about the convergence rate of these algorithms when the $(\alpha, \beta, \delta)$-regularity condition holds. 
In \cite{HXiong_YChi_BHu_WZhang_2020}, the authors showed that under certain parameter settings, accelerated gradient methods converge at a linear rate.
However, due to the techniques used, it is difficult to understand from this paper when this linear rate substantially improves over the convergence rate of gradient descent, leaving open interesting future research directions.


\bibliographystyle{siam} 
\bibliography{refs-strictsaddle}

\appendix
\section{Proof of \eqref{eq:ut1}} \label{app:matrixeqs}
By the definition of $W^*$ and the operator norm:
\begin{align*}
\|W^*\|^2 = \lambdamax(W^* (W^*)^\top)
&= \max_z \,  \frac{z^\top W^* (W^*) z}{\|z\|^2} \\
&= \max_{z_{\Phi},z_{\Psi}} \, \frac{\left[\begin{matrix}
z_{\Phi} \\ z_{\Psi} 
\end{matrix}\right]^\top \left[
\begin{matrix}
\Phi \Sigma \Phi^\top & \Phi \Sigma \Psi^\top \\
\Psi \Sigma \Phi^\top & \Psi \Sigma \Psi^\top
\end{matrix}
\right] \left[\begin{matrix}
z_{\Phi} \\ z_{\Psi} 
\end{matrix}\right]}{\|z_{\Phi}\|^2+ \|z_{\Psi}\|^2},
\end{align*}
where we have partitioned the vector $z$ in an obvious way. Letting $\phi_1$ denote the first left singular vector of $X^*$ and $\psi_1$ denote the first right singular vector of $X^*$. Then, it is clear that the maximum is obtained by setting $z_{\Phi} = \phi_1$ and $z_{\Psi} = \psi_1$, 
so that
\[
\max_{z_{\Phi},z_{\Psi}} \, \frac{\left[\begin{matrix}
z_{\Phi} \\ z_{\Psi} 
\end{matrix}\right]^\top \left[
\begin{matrix}
\Phi \Sigma \Phi^\top & \Phi \Sigma \Psi^\top \\
\Psi \Sigma \Phi^\top & \Psi \Sigma \Psi^\top
\end{matrix}
\right] \left[\begin{matrix}
z_{\Phi} \\ z_{\Psi} 
\end{matrix}\right]}{\|z_{\Phi}\|^2+ \|z_{\Psi}\|^2}
= \frac{4 \sigma_1(X^*)}{2} = 2 \sigma_1(X^*) = 2 \|X^*\|.
\]
To prove the second result, the definition of the Frobenius norm
gives
\[
\|W^* (W^*)^\top\|_F
= \sqrt{\sum_{i=1}^r \lambda_i \left(W^* (W^*)^\top\right)^2}.
\]
Now, let $\psi_i$ be the $i$-th left singular vector of $X^*$ and $\phi_i$ 
be the $i$-th right singular vector of $X^*$. It is clear that we obtain an
eigenvector for  the $i$th eigenvalue of $W^* (W^*)^\top$ by setting
$z_i = \left[\begin{matrix} \phi_i \\ \psi_i \end{matrix}\right]$. Similar to
the calculation above for $z_1$, we have
\[
\lambda_i \left(W^* (W^*)^\top\right) = \frac{z_i^\top W^* (W^*)^\top z_i}{\|z_i\|^2}
= 2 \sigma_i(X^*)
\]
and thus
\[
\sqrt{\sum_{i=1}^r \lambda_i \left(W^* (W^*)^\top\right)^2}
= \sqrt{\sum_{i=1}^r (2 \sigma_i(X^*))^2} =
2 \sqrt{\sum_{i=1}^r \sigma^2_i(X^*)} =
2 \|X^*\|_F.
\]

\end{document}